\documentclass[letterpaper, reqno, 11pt]{amsart}
\usepackage[english]{babel}
\usepackage{graphicx,color}
\usepackage[all]{xy}
\usepackage{amsfonts,dsfont,mathrsfs}
\usepackage{amsmath, amsthm, amssymb}
\usepackage{enumerate, url}
\usepackage[foot]{amsaddr}
\usepackage[pdftex,
                paper=a4paper,
                textheight=675pt]{geometry}

\newcommand{\bbC}{\ensuremath{\mathbb{C}}}
\newcommand{\bbR}{\ensuremath{\mathbb{R}}}

\newcommand{\bbZ}{\ensuremath{\mathbb{Z}}}

\newcommand{\bbB}{\ensuremath{\mathbb{B}}}
\newcommand{\frg}{\ensuremath{\mathfrak{g}}}
\newcommand{\frh}{\ensuremath{\mathfrak{h}}}

\newcommand{\fro}{\ensuremath{\mathfrak{o}}}

\DeclareMathOperator{\id}{id}

\DeclareMathOperator{\rk}{rk}

\DeclareMathOperator{\vol}{vol}
\DeclareMathOperator{\isom}{Isom}
\DeclareMathOperator{\Aut}{Aut}
\DeclareMathOperator{\Inn}{Inn}
\DeclareMathOperator{\Out}{Out}

\DeclareMathOperator{\PSL}{PSL}

\DeclareMathOperator{\SO}{SO}
\DeclareMathOperator{\PSO}{PSO}

\DeclareMathOperator{\SU}{SU}

\numberwithin{equation}{section}

\newtheorem{thmnr}{Theorem}[section]

\newtheorem{propnr}[thmnr]{Proposition}

\newtheorem{lemnr}[thmnr]{Lemma}

\newtheorem{cornr}[thmnr]{Corollary}
\theoremstyle{definition}

\newtheorem{rmknr}[thmnr]{Remark}

\newtheorem{exnr}[thmnr]{Example}

\newtheorem{claimnr}[thmnr]{Claim}

\newtheorem{quesnr}[thmnr]{Question}

\newcounter{tmp}

\begin{document}

\title{Isometry types of frame bundles}
\author{Wouter van Limbeek}
\email{wouterv@umich.edu}
\address{Department of Mathematics, University of Michigan, Ann Arbor, MI 48109.}

\date{\today}

\begin{abstract} We consider the oriented orthonormal frame bundle $\SO(M)$ of an oriented Riemannian manifold $M$. The Riemannian metric on $M$ induces a canonical Riemannian metric on $\SO(M)$. We prove that for two closed oriented Riemannian $n$-manifolds $M$ and $N$, the frame bundles $\SO(M)$ and $\SO(N)$ are isometric if and only if $M$ and $N$ are isometric, except possibly in dimensions 3, 4, and 8. This answers a question of Benson Farb except in dimensions 3, 4, and 8.\end{abstract}

\maketitle
\tableofcontents

\section{Introduction}
\label{sec:intr}
Let $M$ be an oriented Riemannian manifold, and let $X:=SO(M)$ be the oriented orthonormal frame bundle of $M$. The Riemannian structure $g$ on $M$ induces in a canonical way a Riemannian metric $g_{\SO}$ on $\SO(M)$. This construction was first carried out by O'Neill \cite{oneillsasaki} and independently by Mok \cite{mok}, and is very similar to Sasaki's construction of a metric on the unit tangent bundle of $M$ \cite{sasaki1,sasaki2}, so we will henceforth refer to $g_{\SO}$ as the \emph{Sasaki-Mok-O'Neill metric} on $\SO(M)$. Let us sketch the construction of $g_{\SO}$ and refer to Section \ref{sec:prelim} for the details. Consider the natural projection $\pi:SO(M)\rightarrow M$. Each of the fibers of $p$ is naturally equipped with a free and transitive $\SO(n)$-action, so that this fiber carries an $\SO(n)$-bi-invariant metric $g_\mathcal{V}$. The metric $g_{\mathcal{V}}$ is determined uniquely up to scaling. Further, the Levi-Civita connection on the tangent bundle $TM\rightarrow M$ induces a horizontal subbundle of $TM$. This in turn induces a horizontal subbundle $\mathcal{H}$ of $TSO(M)$. We can pull back the metric on $M$ along $\pi$ to get a metric $g_\mathcal{H}$ on $\mathcal{H}$. The Sasaki-Mok-O'Neill metric on $\SO(M)$ is defined to be $g_\mathcal{\SO}:=g_\mathcal{V}\oplus g_\mathcal{H}$.

Note that $g_{\SO}$ is determined uniquely up to scaling of $g_\mathcal{V}$, and hence determined uniquely after fixing a bi-invariant metric on $\SO(n)$. The work of O'Neill \cite{oneillsasaki}, Mok \cite{mok}, and later Takagi-Yawata \cite{isomframe1,isomframe2} have established many natural properties of Sasaki-Mok-O'Neill metrics and connections between the geometry of $M$ and $\SO(M)$. The following natural question then arises, which was to my knowledge first posed by Benson Farb.
	\begin{quesnr} Let $M,N$ be Riemannian manifolds. If $\SO(M)$ is isometric to $\SO(N)$ (with respect to Sasaki-Mok-O'Neill metrics on each), is $M$ isometric to $N$?\label{ques:main}\end{quesnr}
The purpose of this paper is to answer Question \ref{ques:main} except when $\dim M=3, 4$ or $8$. The question is a bit subtle, for it is not true in general that an isometry of $\SO(M)$ preserves the fibers of $\SO(M)\rightarrow M$ as shown by the following example. 
\begin{exnr} Let $M$ be a constant curvature sphere $S^n$. Then $\SO(M)$ is diffeomorphic to $\SO(n+1)$. (To see this, identify $S^n$ with the unit sphere in $\bbR^{n+1}$. If $p\in S^n$ and $v_1,\dots,v_n$ is a positively oriented orthonormal frame at $p$, then the matrix with columns $p,v_1,\dots,v_n$ belongs to $\SO(n+1)$.) There is a unique Sasaki-Mok-O'Neill metric that is isometric to the bi-invariant metric on $\SO(n+1)$. However, of course there are many isometries of $\SO(n+1)$ that do not preserve the fibers of $\SO(n+1)\rightarrow S^n$. \label{ex:nonpres}\end{exnr}
By differentiating the action of $\SO(n+1)$ in the above example, we obtain many Killing fields that do not preserve the fibers of $\SO(n+1)\rightarrow S^n$. However, by a theorem of Takagi-Yawata \cite{isomframe1}, manifolds with constant positive curvature are the only Riemannian manifolds whose orthonormal frame bundles admit Killing fields that do not preserve the fibers. More examples of non-fiber-preserving isometries appear if we consider isometries that are not induced by Killing fields, as the following example shows.
\begin{exnr} Let $M$ be a flat 2-torus obtained as the quotient of $\bbR^2$ by the subgroup generated by translations by $(l_1, 0)$ and $(0, l_2)$ for some $l_1, l_2>0$. Further fix $l_3>0$ and equip $\SO(M)$ with the Sasaki-Mok-O'Neill metric associated to the scalar $l_3$. It is easy to see $\SO(M)$ is the flat 3-torus obtained as the quotient of $\bbR^3$ by the subgroup generated by translations by $(l_1, 0, 0), (0, l_2, 0)$ and $(0, 0, l_3)$.

Now let $N$ be the flat 2-torus obtained as the quotient of $\bbR^2$ by the subgroup generated by translations by $(l_1, 0)$ and $(0, l_3)$, and equip $\SO(N)$ with the Sasaki-Mok-O'Neill metric associated to the scalar $l_2$. Then $\SO(M)$ and $\SO(N)$ are isometric but if $l_1, l_2, l_3$ are distinct, $M$ and $N$ are not isometric. 

On the other hand if $l_1=l_3\neq l_2$, then this construction produces an isometry $\SO(M)\rightarrow SO(M)$ that is not a bundle map. \label{ex:swaptori}\end{exnr}

Example \ref{ex:swaptori} produces counterexamples to Question \ref{ques:main}. Note that we used different bi-invariant metrics $g_\mathcal{V}$ on the fibers. Therefore to give a positive answer to Question \ref{ques:main} we must normalize the volume of the fibers of $\SO(M)\rightarrow M$.

Our main theorem is that under the assumption of normalization Question \ref{ques:main} has the following positive answer, except possibly in dimensions $3, 4$ and $8$.

\begingroup
\setcounter{tmp}{\value{thmnr}}% store current value of theorem counter
\setcounter{thmnr}{0} %assign desired value to theorem counter
\renewcommand\thethmnr{\Alph{thmnr}}% locally redefine the representation of the theorem counter
	\begin{thmnr}Let $M,N$ be closed oriented connected Riemannian $n$-manifolds. Equip $\SO(M)$ and $\SO(N)$ with Sasaki-Mok-O'Neill metrics where the fibers of $\SO(M)\rightarrow M$ and $\SO(N)\rightarrow N$ have fixed volume $\nu>0$. Assume $n\neq 3,4,8$. Then $M,N$ are isometric if and only if $\SO(M)$ and $\SO(N)$ are isometric.
	\label{thm:main}\end{thmnr}
	\endgroup
\setcounter{thmnr}{\thetmp} % restore value of theorem counter
We do not know if counterexamples to Question \ref{ques:main} exist in dimensions 3, 4, and 8.
\subsection*{Outline of proof} If $f:M\rightarrow N$ is an isometry, then the induced map $\SO(f):SO(M)\rightarrow SO(N)$ is also an isometry (see Proposition \ref{prop:isomisom}). This proves one direction of the theorem.

For the other direction, our strategy is to identify the fibers of the bundle $\SO(M)\rightarrow M$ using only the geometry of $\SO(M)$. To accomplish this, note that $X=SO(M)$ carries an action of $\SO(n)$ by isometries, and the orbits of this action are exactly the fibers of $\SO(M)\rightarrow M$. This action gives rise to an algebra of Killing fields isomorphic to $\mathfrak{o}(n)$. 

The full Lie algebra $i(X)$ of Killing fields on $X=SO(M)$ has been computed by Takagi-Yawata \cite{isomframe2} except in dimensions 2, 3, 4 or 8, or when $M$ has positive constant curvature. We show that if this computation applies, either $i(X)$ contains a unique copy of $\mathfrak{o}(n)$ or Isom$(M)$ is extremely large or $M$ is flat. If $i(X)$ contains a unique copy of $\mathfrak{o}(n)$, then the fibers of $X=SO(M)\rightarrow M$ and $X=SO(N)\rightarrow N$ coincide, and we deduce that $M$ and $N$ are isometric.

We are able to resolve the flat case separately. If Isom$(M)$ is large we use classifcation theorems from the theory of compact transformation groups to prove that $M$ and $N$ are isometric. 

Finally we prove the theorem in two situations where the computation of Takagi-Yawata does not apply, namely constant positive curvature and dimension 2. In these situations it is in general impossible to identify the fibers of $\SO(M)\rightarrow M$ using the geometry of $\SO(M)$ alone as shown by Examples \ref{ex:nonpres} and \ref{ex:swaptori}. However, we are still able to obtain the main result using the scarcity of manifolds with a metric of constant positive curvature, and the classification of surfaces.

\subsection*{Outline of the paper} In Section \ref{sec:act} we will review preliminaries about actions of Lie groups of $G$ on a manifold $M$ when $\dim G$ is large compared to $\dim M$. In Section \ref{sec:main} we will prove the Main Theorem \ref{thm:main} except when $M$ and $N$ are surfaces or have metrics of consnt positive curvature. The proof in the case that at least one of $M$ or $N$ has constant positive curvature will be given in Section \ref{sec:kcnst}. We prove Theorem \ref{thm:main} in the case that $M$ and $N$ are surfaces in Section \ref{sec:surf}.
\subsection*{Acknowledgments} I am very grateful to my thesis advisor Benson Farb for posing Question \ref{ques:main} to me, his extensive comments on an earlier version of this paper, and his invaluable advice and enthusiasm during the completion of this work. I am indebted to an anonymous referee for many helpful suggestions. I would like to thank the University of Chicago for support.

\section{Preliminaries}\label{sec:prelim}

In this section we introduce the Sasaki-Mok-O'Neill metric, and we recall some basic properties. Then we discuss the classical relationship between isometries and Killing fields, and Takagi-Yawata's computations of Killing fields of Sasaki-Mok-O'Neill metrics. We end this section with a useful lemma for normalizing Sasaki-Mok-O'Neill metrics, and some general remarks about frame bundles of fiber bundles that will also be useful later.

\subsection{Definition of the Sasaki-Mok-O'Neill metric.} Our discussion here follows the construction of Mok \cite{mok}, where more details can be found. Let $(M,g)$ be an oriented Riemannian manifold of dimension $n$, and let $X:=\SO(M)$ be the oriented orthonormal frame bundle of $M$ with natural projection map $\pi:\SO(M)\rightarrow M$. For $e\in \SO(M)$, the \emph{vertical subspace} at $e$ is defined to be $\mathcal{V}_e:=\ker D_e\pi$. The collection of vertical subspaces forms a subbundle $\mathcal{V}\rightarrow TM$ of $T\SO(M)\rightarrow TM$.

Let $\omega$ be the Riemannian connection $\fro(n)$-valued 1-form associated to the Riemannian metric on $M$. Explicitly, if $p\in M$ and $e=(e_1,\dots,e_n)$ is a frame at $p$, we define for $X\in T_e SO(M)$:
$$\omega_{ij}(X):=\theta_j(\nabla_X(e_i)) \hspace{1.5 cm} (1\leq i,j\leq n),$$
where $\theta_j$ is the form dual to $e_j$ with respect to the Riemannian metric $g$. 

We set $\mathcal{H}_e:=\ker\omega_e$. We call $\mathcal{H}_e$ the \emph{horizontal subspace} at $e$. We have a decomposition $T_e\SO(M)=\mathcal{V}_e\oplus\mathcal{H}_e$. Define an inner product on $T_e \SO(M)$ via
	$$g_{\SO}(X,Y)=\langle \omega(X),\omega(Y)\rangle+g(\pi_\ast X,\pi_\ast Y)$$
where $\langle\cdot , \cdot \rangle$ is an $O(n)$-invariant inner product on $\fro(n)$. Note that the choice of an $O(n)$-invariant inner product on $\fro(n)$ is uniquely determined up to scaling by a positive number $\lambda$, so that we obtain a 1-parameter family of Sasaki-Mok-O'Neill metrics. Explicitly such an inner product is given by
	$$\langle A, B\rangle_\lambda := -\lambda~\textrm{tr}(AB)= \lambda \sum_{i,j} A_{ij} B_{ij}$$
for $A,B\in\fro(n)$. We call $\langle \cdot, \cdot \rangle_1$ the \emph{standard metric} on $\fro(n)$.

\begin{rmknr} The oriented orthonormal frame bundle $\SO(M)\rightarrow M$ is an example of a $\SO(n)$-principal bundle of over $M$, and it has a natural connection form $\omega$ as defined above. For a principal $G$-bundle $E\rightarrow B$ with a principal connection form $\theta$, one can construct a so-called connection metric (see e.g. \cite[Section 1]{zillerfat}). The Sasaki-Mok-O'Neill metric is exactly this connection metric in the case of the principal $\SO(n)$-bundle $\SO(M)\rightarrow M$ with the connection form $\omega$. \end{rmknr}

As mentioned above, the geometry of the above-defined metric was first investigated by O'Neill and Mok. In particular they showed:
	\begin{propnr}[{O'Neill \cite[p. 467]{oneillsasaki}, Mok \cite[Theorem 4.3]{mok}}]\label{prop:fibertotgeod} The fibers of $\SO(M)\rightarrow M$ are totally geodesic submanifolds of $\SO(M)$ with respect to any Sasaki-Mok-O'Neill metric.\end{propnr}

\subsection{Vector fields on frame bundles.} Let $X$ be a vector field on $\SO(M)$. If $X_e\in \mathcal{V}_e$ for any $e\in\SO(M)$, we say $X$ is \emph{vertical}. If $X_e\in \mathcal{H}_e$ for any $e\in \SO(M)$, we say $X$ is \emph{horizontal}. 

We will now discuss how to lift a vector field $Y$ on $M$ to a vector field $X$ on $\SO(M)$ such that $\pi_\ast X= Y$. There are two useful constructions, called the \emph{horizontal} and \emph{complete} lift of $Y$. Both constructions start by considering the derivative of the bundle map $\pi:\SO(M)\rightarrow M$. For a frame $e\in \SO(M)$, we have a decomposition $T_e \SO(M)=\mathcal{V}_e\oplus \mathcal{H}_e$ as discussed above. Here $\mathcal{V}_e = \ker\pi_\ast$, and hence $\pi_\ast$ restricts to an isomorphism $\mathcal{H}_e\rightarrow T_{\pi(e)}M$. Therefore for a vector field $Y$ on $M$, there exists a unique horizontal vector field $Y^H$ on $M$ with $Y=\pi_\ast Y^H$. We call $Y^H$ the \emph{horizontal lift} of $Y$.

The \emph{complete lift} $Y^C$ of $Y$ was first introduced by Kobayashi-Nomizu  \cite{KN63}. First observe that given a map $f:M\rightarrow M$, we can consider its induced map $\SO(f):\SO(M)\rightarrow \SO(M)$ on frames. Then we can define $Y^C$ as follows: Let $\varphi_t$ be the 1-parameter family of diffeomorphims of $M$ obtained by integrating $Y$, so that $Y=\frac{d}{dt}\Bigr|_{\substack{t=0}} \varphi_t$. Then we define
	$$Y^C:=\frac{d}{dt}\Bigr|_{\substack{t=0}} \SO(\varphi_t).$$
Note that $Y^C$ is in general neither vertical nor horizontal. Mok has given a description of $Y^C$ in terms of local coordinates \cite[Section 3]{moklift}.
		
\subsection{Killing fields and isometries.} Before considering the isometries of $\SO(M)$ equipped with a Sasaki-Mok-O'Neill metric $g_{\SO}$, we will review some classical facts about the structure of the group of isometries Isom$(M)$ of a Riemannian manifold $M$.

Myers-Steenrod \cite{myerssteenrod} have proved that Isom$(M)$ of a Riemannian manifold is a Lie group. If $(h_t)_t$ is a 1-paremeter group of isometries, then $Y:=\frac{d}{dt}\Bigr|_{\substack{t=0}} h_t$ is a vector field on $M$. Differentiating the condition $h_t^\ast g = g$ gives the \emph{Killing relation} for $Y$: 
	\begin{equation} \label{eq:killing} \mathcal{L}_Y g =0,\end{equation}
where $\mathcal{L}$ is the Lie derivative. Any vector field $Y$ satisfying Equation \ref{eq:killing} is called a Killing field. Given a Killing field $Y$ on $M$, the 1-parameter group $(h_t)_t$ obtained by integrating $Y$ consists of isometries. The Killing fields on $M$ form a Lie algebra $i(M)$ of vector fields. We have:
	\begin{thmnr}\label{thm:isomkilling} Let $M$ be a Riemannian manifold. Then \emph{Isom}$(M)$ is a Lie group (possibly not connected), with Lie algebra $i(M)$.\end{thmnr}	 

\subsection{The Takagi-Yawata theorem on Killing fields.}\label{sec:taya} We will now discuss a complete description due to Takagi-Yawata \cite{isomframe2} of the Killing fields on $\SO(M)$ in terms of the geometry of $M$ for many manifolds $M$. Let us first discuss three constructions of Killing fields on $\SO(M)$.

For the first construction, recall that Sasaki showed (see \cite[Corollary 1]{sasaki1}) that whenever $f:M\rightarrow M$ is an isometry of $M$, the derivative $Df: TM\rightarrow TM$ is an isometry of $TM$ (where $TM$ is equipped with a Sasaki metric). Therefore if $Y$ is a Killing field on $M$, then the complete lift of $Y$ is a Killing field on $TM$. This is also true for frame bundles:
	\begin{propnr}[{Mok \cite[Proposition 5.3]{mok}}] If $Y$ is a Killing field on $M$, then $Y^C$ is a Killing field on $\SO(M)$ with respect to any Sasaki-Mok-O'Neill metric.\end{propnr}
In fact the following more general statement is true:
	\begin{propnr}\label{prop:isomisom} Let $M$ be a Riemannian manifold and $f:M\rightarrow M$ any isometry. Then the induced map $\SO(f):\SO(M)\rightarrow\SO(M)$ is an isometry of $\SO(M)$ with respect to any Sasaki-Mok-O'Neill metric.\end{propnr}
	\begin{proof} Note that since the Riemannian connection form $\omega$ is canonically associated to the metric, we have $f^\ast \omega=\omega$. In particular $\SO(f)$ preserves the horizontal subbundle $\mathcal{H}:=\ker\omega$. Also note that $\SO(f)$ is a bundle map of $\pi:\SO(M)\rightarrow M$ (i.e. we have $\SO(f)\circ\pi=\pi\circ f$), and in particular $\SO(f)$ preserves the vertical subbundle $\mathcal{V}:=\ker\pi_\ast$. Using these facts it is easy to check $\SO(f)$ is an isometry.\end{proof}
The second construction of Killing fields comes from the structure of $\SO(M)\rightarrow M$ as a principal $\SO(n)$-bundle. There is an action of $\SO(n)$ on the fibers of $\SO(M)\rightarrow M$, which is easily seen to be isometric with respect to any Sasaki-Mok-O'Neill metric. Differentiating any 1-parameter subgroup of $\SO(n)$ then gives a Killing field on $\SO(M)$. Explicitly, we can define these as follows: For $A\in \fro(n)$, define the vector field $A^\ast$ on $\SO(M)$ via $\omega(A^\ast) = A$ and $\pi_\ast(A^\ast)=0$, where $\omega$ is the connection form as above. Then $A^\ast$ is a vertical Killing field. Write $i_V^M$ for the Killing fields thus obtained. In particular $i_V^M\cong \fro(n)$ as Lie algebras.

Finally, here is the third construction of a Killing field on $\SO(M)$. Let $\varphi$ be a 2-form on $M$, so that it defines a skew-symmetric bilinear form on every tangent space $T_p M$ for  $p\in M$. With respect to a frame $e$ of $T_p M$, the skew-symmetric form $\varphi_p$ can be represented as a skew-symmetric matrix $A_e\in \fro(n)$. We then define a vector field $X^\varphi$ on $\SO(M)$ via $\omega_e(X^\varphi_e):=A_e$ and $\pi_\ast(X^\varphi_e)=0$. Note that the latter condition just means that we define $X^\varphi$ to be a vertical vector field. An explicit computation shows that if $\varphi$ is parallel, then $X^\varphi$ is a Killing field (see e.g. \cite{isomframe1}). Denote by $(\Lambda^2 M)_0$ the Lie algebra of parallel 2-forms on $M$.

 Takagi-Yawata have proved that for many manifolds, the above three constructions are the only ways of producing Killing fields on $\SO(M)$:

\begin{thmnr}[Takagi-Yawata \cite{isomframe2}] Let $M$ be a closed Riemannian manifold and equip $\SO(M)$ with the Sasaki-Mok-O'Neill metric corresponding to the standard inner product $\langle \cdot ,\cdot \rangle_1$ on $\fro(n)$. Suppose $M$ does not have constant curvature $\frac{1}{2}$ and $\dim M\neq 2, 3, 4, 8$. Then for any Killing field $X$ on $\SO(M)$ there exist unique $Y\in i(M), A\in\fro(n)$, and $\varphi\in (\Lambda^2 M)_0$ such that 
	$$X=Y^C + A^\ast + X^\varphi.$$
\label{thm:tyexist}
\end{thmnr}
\begin{rmknr} Of course a version of the above result holds for different Sasaki-Mok-O'Neill metrics as well: If we use the inner product $\langle \cdot, \cdot\rangle_\lambda = \lambda \langle \cdot, \cdot\rangle_1$ on $\fro(n)$, the same conclusion holds except that we should now require that $M$ does not have constant curvature $\frac{1}{2\sqrt{\lambda}}$.\end{rmknr}
An explicit computation shows that if $Y\in i(M), A\in\fro(n)$, and $\varphi\in(\Lambda^2 M)_0$, then the vector fields $Y^C, A^\ast$, and $X^\varphi$ pairwise commute. Combining this with Theorem \ref{thm:tyexist}, we obtain the following Lie algebra decomposition of Killing fields on $\SO(M)$:
\begin{cornr}[Takagi-Yawata \cite{isomframe2}] \label{cor:tydecomp} Let $M$ be a Riemannian manifold satisfying the hypotheses of Theorem \ref{thm:tyexist}. Then there is a Lie algebra decomposition
	$$i(\SO(M))=i(M)\oplus i_V^M \oplus (\Lambda^2 M)_0,$$
where $i(M)$ (resp. $i_V^M, (\Lambda^2 M)_0$) corresponds to the subalgebra of Killing fields consisting of $Y^C$ (resp. $A^\ast, X^\varphi$) for $Y\in i(M)$ (resp. $A\in\fro(n), \varphi\in(\Lambda^2 M)_0$).
\end{cornr}

\subsection{Normalizing volume.} Given a closed oriented Riemannian manifold $M$, we have previously obtained a 1-parameter family of Sasaki-Mok-O'Neill metrics on $M$. These can be parametrized by a choice of $O(n)$-invariant inner product on $\fro(n)$ (which is unique up to scaling), or, equivalently, by the volume of a fiber of $\SO(M)\rightarrow M$. The following easy lemma will be useful to us on multiple occasions in the rest of the paper.

	\begin{lemnr} Fix $\nu>0$. Let $M,N$ be closed orientable connected Riemannian $n$-manifolds and equip $\SO(M)$ and $\SO(N)$ with Sasaki-Mok-O'Neill metrics where the fibers of $\SO(M)\rightarrow M$ and $\SO(N)\rightarrow N$ have volume $\nu$. Suppose that $\SO(M)$ and $\SO(N)$ are isometric. Then $\vol(M)=\vol(N)$.
	\label{lem:vol}
	\end{lemnr}
	
		\begin{proof} Set $X:=SO(M)\cong SO(N)$. Since the fiber bundle $X\rightarrow M$ has fibers with volume $\nu$, we have
			$$\vol(X)=\frac{\vol(M)}{\nu}.$$
		Likewise we have $\vol(X)=\frac{\vol(N)}{\nu}$. Combining these we get $\vol(M)=\vol(N)$.\end{proof}

\section{High-dimensional isometry groups of manifolds}\label{sec:act}
In this section we review some known results about effective actions of a compact Lie group $G$ on a closed $n$-manifold $M$ when $\dim G$ is large compared to $n$. We will be especially interested in actions of $\SO(n)$ on an $n$-manifold $M$. First, there is the following classical upper bound for the dimension of a compact group acting smoothly on an $n$-manifold.
	\begin{thmnr}[{\cite[II.3.1]{kotrgps}}] Let $M$ be a closed $n$-manifold and $G$ a compact group acting smoothly, effectively, and isometrically on $M$. Then $\dim G\leq \frac{n(n+1)}{2}$. Further equality holds if and only if 
		\begin{enumerate}[(i)]
			\item $M$ is isometric to $S^n$ with a metric of constant positive curvature, and $G=\SO(n+1)$ or $\text{O}(n+1)$ acting on $M$ in the standard way, or
			\item $M$ is isometric or $\bbR P^n$ with a metric of constant positive curvature, and $G=\text{PSO}(n+1)$ or $\text{PO}(n+1)$, acting on $M$ in the standard way.
		\end{enumerate}
	\label{thm:maxisom}
	\end{thmnr}
Note that in the above case $G=\textrm{Spin}(n+1)$ does not occur because there is no effective action on $S^n$ or $\bbR P^n$. Theorem \ref{thm:maxisom} leads us to study groups of dimension $<\frac{n(n+1)}{2}$. First, there is the following remarkable `gap theorem' due to H.C. Wang.
	\begin{thmnr}[H.C. Wang {\cite{wanggaps}}] Let $M$ be a closed $n$-manifold with $n\neq 4$. Then there is no compact group $G$ acting effectively on $M$ with
		$$\frac{n(n-1)}{2}+1<\dim G<\frac{n(n+1)}{2}.$$
		\label{thm:wanggaps}
	\end{thmnr}
Therefore the next case to consider is $\dim G=\frac{n(n-1)}{2}+1$. The following characterization is independently due to Kuiper and Obata.
	\begin{thmnr}[Kuiper, Obata {\cite[II.3.3]{kotrgps}}] Let $M$ be a closed Riemannian $n$-manifold with $n>4$ and $G$ a connected compact group of dimension $\frac{n(n-1)}{2}+1$ acting smoothly, effectively, and isometrically on $M$. Then $M$ is isometric to $S^{n-1}\times S^1$ or $\bbR P^{n-1}\times S^1$ equipped with a product of a round metric on $S^{n-1}$ or $\bbR P^{n-1}$ and the standard metric on $S^1$. Further $G=SO(n)\times S^1$ or $PSO(n)\times S^1$.
	\label{thm:seccase}
	\end{thmnr}
After Theorem \ref{thm:seccase}, the natural next case to consider is $\dim G=\frac{n(n-1)}{2}$. There is a complete classification due to Kobayashi-Nagano \cite{kobnag}. 
	\begin{thmnr}[Kobayashi-Nagano] Let $M$ be a closed Riemannian $n$-manifold with $n>5$ and $G$ a connected compact group of dimension $\frac{n(n-1)}{2}$ acting smoothly, effectively, and isometrically on $M$. Then $M$ must be one of the following.
		\begin{enumerate}
			\item $M$ is diffeomorphic to $S^n$ or $\bbR P^n$ and $G=SO(n)$ or $PSO(n)$. In this case $G$ has a fixed point on $M$. Every orbit is either a fixed point or has codimension 1. Regarding $S^n$ as the solution set of $\sum_{i=0}^n x_i^2 =1$ in $\bbR^{n+1}$, the metric on $M$ (or its double cover if $M$ is diffeomorphic to $\bbR P^n$) is of the form
				$$ds^2 = f(x_0) \sum_{i=0}^n dx_i^2$$
			for a smooth positive function $f$ on $[-1,1]$. 
			\item $M$ is diffeomorphic to a quotient $(L\times \bbR)\slash \Gamma$ where $L=S^{n-1}$ or $L=\bbR P^{n-1}$ and $G=\SO(n)$ or $\PSO(n)$. Further, we have $\Gamma\cong \bbZ$. If $L=S^{n-1}$, then $\Gamma$ is generated either by the map $(v,t)\mapsto (v,t+1)$ or by $(v,t)\mapsto (-v,t+1)$. If $L=\bbR P^{n-1}$, then $\Gamma$ is generated by the map $(x,t)\mapsto (x,t+1)$. In all cases the projection on the second coordinate $S^{n-1}\times \bbR\rightarrow \bbR$ descends to a map $M\rightarrow S^1$ that is a fiber bundle with fibers diffeomorphic to $L$. The $G$-action preserves the fibers of $M\rightarrow S^1$ and restricts to an orthogonal action on each fiber.
			\item $M$ is a quotient $(S^{n-1}\times \bbR)\slash\Gamma$ where $\Gamma$ is generated by
				\begin{align*}
					(v,t)\mapsto (v,t+2)\\
					(v,t)\mapsto (-v,-t).
				\end{align*}
			In this case $G=SO(n)$ acts on $S^{n-1}\times \bbR$ by acting orthogonally on each copy $S^{n-1}\times\{t\}$. This action commutes with the action of $\Gamma$, so that the $G$-action descends to $M$. We have $M\slash G=[0,1]$. The $G$-orbits lying over the endpoints 0,1 are isometric to round projective spaces $\bbR P^{n-1}$ and the $G$-orbits lying over points in $(0,1)$ are round spheres.
			\item If $n=6$ there is the additional case that $M\cong \bbC P^3$, equipped with the Fubini-Study metric and the standard action of $G=SO(6)\cong SU(4)\slash \{\pm \id\}$.
			\item If $n=7$ there are the additional cases $M\cong \textrm{Spin}(7)\slash G_2$ and $G=\textrm{Spin}(7)$, or $M\cong SO(7)\slash G_2$ and $G=SO(7)$. In this case $M$ is isometric to $S^7$ or $\bbR P^7$ with a constant curvature metric. 
			\end{enumerate}
		\label{thm:3case}
	\end{thmnr}
	\begin{rmknr} Actually Kobayashi-Nagano prove a more general result that includes the possibility that $M$ is noncompact, and there are more possibilities. Since we will not need the noncompact case, we have omitted these. In their formulation of Case (4), $M$ is a manifold 
	of complex dimension 3 with constant holomorphic sectional curvature, and $G$ is
	the largest connected group of holomorphic isometries. 
	
	Specializing to the compact case gives an explicit description of Case (4) as follows. Hawley \cite{holsec1} and Igusa \cite{holsec2} independently proved that a simply-connected complex $n$-manifold of constant holomorphic sectional curvature is isometric to either $\bbC^n, \bbB^n$ or $\bbC P^n$ (with standard metrics). Therefore in Case (4) we obtain that $M$ is isometric to $\bbC P^3$ (equipped with a scalar multiple of the Fubini-Study metric) and $G=SO(6)\cong SU(4)\slash\{\pm \id\}$.
	\label{rmk:holsec} \end{rmknr}
	\begin{rmknr} If $M$ admits the description in Case (2) above, and is in addition assumed to be orientable, it follows that the bundle $M\rightarrow S^1$ is trivial. In particular $M$ is diffeomorphic to $L\times S^1$. 
	
	To see this,  note that the only other case to consider is that $M=(S^{n-1}\times \bbR)\slash \Gamma$ where $\Gamma\cong \bbZ$ is generated by the map $(v,t)\mapsto (-v,t+1)$. This is a bundle with monodromy $-\id \in \text{Diff}(S^{n-1})$. Two bundles over $S^1$ are equivalent if and only if their monodromies are isotopic (i.e. belong to the same component of Diff$(S^{n-1})$. So let us check that $-\id$ is isotopic to the identity map: Indeed, because $M$ is orientable, the map $(v,t)\mapsto (-v,t+1)$ is orientation preserving on $S^{n-1}\times \bbR$. It follows that $n$ is even, so that $-\id\in \SO(n)$ and hence is clearly isotopic to the identity map.
	\label{rmk:orient}\end{rmknr} 
		
Theorem \ref{thm:3case} does not cover the case $n=5$. In the following proposition we resolve this case for semisimple groups. We would like to thank an anonymous referee for the following statement and its proof, which improve upon those contained in an earlier version of this paper.
	\begin{propnr} Let $M$ be a closed oriented Riemannian 5-manifold and suppose $G$ is a semisimple compact connected Lie group that acts on $M$ smoothly, effectively, and isometrically, and that $\dim(G)=10$. Then $M$ admits a description as in Cases (1), (2) or (3) of Theorem \ref{thm:3case}.\label{prop:3case}\end{propnr}
	\begin{proof} The proof of Theorem \ref{thm:3case} (see \cite[Section 3]{kobnag}) shows that the assumption that $n>5$ is only used to show that no $G$-orbit has codimension 2. We will show under the stated assumptions there are still no codimension 2 orbits, so that the rest of the proof of Theorem \ref{thm:3case} applies.
	
	Clearly we can assume that $G$ is connected. Note that $\dim(G)=\textrm{rank}(G)+2k$, where $k$ is the number of root spaces of $G$. Hence the rank of $G$ is even. Any semisimple Lie group with rank $\geq 4$ has dimension $>10$, so that we must have that rank$(G)=2$ and therefore $G$ is a quotient of $\textrm{Spin}(5)$. 
	
	Suppose now that $x\in M$ and that the orbit $G(x)$ has codimension 2 in $M$. Let $G_x$ be the stabilizer of $x$. Note that $G_x$ has rank either 1 or 2, and since the orbit of $x$ is codimension 2, we must have that $\dim G_x =7$. 
	
	If $G_x$ has rank 1, then it must be $S^1$ or $\textrm{Spin}(3)$ (possibly up to a finite quotient), but then we see that $\dim G_x<7$, so this is impossible.
	
	On the other hand if rank($G_x$)=2, then the dimension of $G_x$ is even, which is also a contradiction.\end{proof}

\section{Geometric characterization of the fibers of $\SO(M)\rightarrow M$}
\label{sec:main}

We will now start the proof of Theorem A.  In this section we aim to prove the following theorem, which proves Theorem A in all cases except for round spheres and surfaces. The remaining cases are resolved in Section \ref{sec:kcnst} (round spheres) and Section \ref{sec:surf} (surfaces).
	\begin{thmnr} Let $M,N$ be closed oriented connected Riemannian $n$-manifolds and fix $\lambda>0$. Equip $\SO(M)$ and $\SO(N)$ with Sasaki-Mok-O'Neill metrics using the metric $\langle \cdot, \cdot\rangle_\lambda$ on $\fro(n)$. Assume that $n\neq 2, 3,4,8$ and that $M$ does not have constant curvature $\frac{1}{2\sqrt{\lambda}}$. Then $M,N$ are isometric if and only if $\SO(M)$ and $\SO(N)$ are isometric.
	\label{thm:mainmost}
	\end{thmnr}
\begin{proof} Write $X:=SO(M)\cong SO(N)$, and let 
	\begin{eqnarray*}
		\pi_M: X\rightarrow M\\
		\pi_N: X\rightarrow N
	\end{eqnarray*}
be the natural projections. The strategy of the proof is to characterize the fibers of $\pi_M$ and $\pi_N$ just in terms of the geometry of $X$, except when $M$ is flat or Isom$(M)$ has dimension at least $\frac{1}{2}n(n-1)$. It automatically follows that in all but the exceptional cases the fibers of $\pi_M$ and $\pi_N$ must agree, and we will use this to show that $M$ and $N$ are isometric. Finally we will show that in the exceptional cases $M$ and $N$ also have to be isometric.

Note that the assumptions of Theorem \ref{thm:mainmost} guarantee that we can use Takagi-Yawata's computation of the Lie algebra of Killing fields on $X$, so we can write (see Corollary \ref{cor:tydecomp}):
	$$i(X)=i(M)\oplus i_V^M\oplus (\Lambda^2 M)_0.$$
Here, as before, $i(M)$ denotes the space of Killing fields on $M$, and $i_V^M$ consists of the Killing fields $A^\ast$ for $A\in\fro(n)$ (in particular $i_V^M\cong \fro(n)$), and $(\Lambda^2 M)_0$ denotes the space of parallel 2-forms on $M$.
	%Warning: Takagi and Yawata state that the Lie algebra $i(X)$ is the direct sum of the subalgebras $(\Lambda^2 M)_0$ and $i(M)$ and $\mathfrak{o}(n)$, but this is false, which can immediately be seen from the statement of their main theorem.
On the other hand, the natural action of $\SO(n)$ on the fibers of $\pi_N$ induces an embedding of $\SO(n)$ in Isom($X$), hence an embedding of Lie algebras
	$$\fro(n)\cong i_V^N\hookrightarrow i(X)=i_V^M \oplus (\Lambda^2 M)_0 \oplus i(M).$$
We identify $i_V^N$ with its image throughout. Now consider the projections of $i_V^N$ onto each of the factors of this decomposition. We have the following cases:
	\begin{enumerate}
		\item $i_V^N = i_V^M$, or
		\item $i_V^N$ projects nontrivially to $(\Lambda^2 M)_0$, or
		\item $i_V^N$ projects trivially to $(\Lambda^2 M)_0$ but nontrivially to $i(M)$.
	\end{enumerate}
We will show below that these cases correspond to (1) the fibers of $\pi_M$ coincide with the fibers of $\pi_N$, (2) $M$ is flat, and (3) $\dim \textrm{Isom}(M)\geq \frac{1}{2}n(n-1)$. We will complete the proof of Theorem \ref{thm:mainmost} in each of these cases below.
\subsection*{Case 1 (vertical directions agree)} Assume that $i_V^N=i_V^M$. For any $x\in X$, the values of $i_V^M$ at $x$, i.e. the set of vectors
	$$\{Z(x)\mid Z\in i_V^M\},$$
span the tangent space to the fiber of $\pi_M$ through $x$. On the other hand this set also spans the tangent space to the fiber of $\pi_N$ through $x$. It follows that the fibers of $\pi_M$ and $\pi_N$ actually coincide. Hence we have a natural map $f:M\rightarrow N$ defined as follows: For $p\in M$, let $x\in \pi_M^{-1}(p)$ be any point in the fiber of $\pi_M$ over $p$. Then set $f(p):=\pi_N(x)$. The fact that the fibers of $\pi_M$ and $\pi_N$ coincide proves that $f(p)$ does not depend on the choice of $x$. 

We claim $f$ is an isometry. Denote by $\mathcal{H}^M$ and $\mathcal{V}^M$ the horizontal and vertical subbundles with respect to $\pi_M: X\rightarrow M$. Because $\pi_M$ is a Riemannian submersion, the metric on $T_x M$ coincides with the metric on the horizontal subbundle $\mathcal{H}_u^M$ at a point $u\in \pi_M^{-1}(x)$. We have 
	$$\mathcal{H}_u^M = (\mathcal{V}_u^M)^\perp=(\ker(\pi_M)_\ast)^\perp =(\ker (\pi_N)_\ast)^\perp.$$
Here the first identity is because by definition of the Sasaki-Mok-O'Neill metric on $X$, the horizontal and vertical subbundles are orthogonal. The last identity follows because we know the fibers of $\pi_M$ and $\pi_N$ agree. Finally, note that the space $(\ker (\pi_N)_\ast)^\perp$ is just the horizontal subbundle of $\pi_N:X\rightarrow N$. Since $\pi_N$ is a Riemannian submersion, we conclude that the metric on $\mathcal{H}_u^M$ coincides with the metric on $T_{\pi_N(u)}N$. This proves the naturally induced map
	$$f:M\rightarrow N$$
is a local isometry. Since $f$ is also injective, $M$ and $N$ are isometric.
\subsection*{Case 2 (many parallel forms)} Assume that $i_V^N\cong\fro(n)$ projects nontrivially to $(\Lambda^2 M)_0$. Note that the kernel of the projection of $i_V^N$ to $(\Lambda^2 M)_0$ is an ideal in $i_V^N$. On the other hand $i_V^N\cong \fro(n)$ is simple (because $n>4$), so the projection $i_V^N\rightarrow (\Lambda^2 M)_0$ must be an isomorphism onto its image. Therefore 
	\begin{equation}
	\dim (\Lambda^2 M)_0\geq \dim \mathfrak{o}(n)=\frac{n(n-1)}{2}.
	\label{eq:dim}
	\end{equation}
We claim that we actually have equality in Equation \ref{eq:dim}. To see this, note that since a parallel form is invariant under parallel transport, it is determined by its values on a single tangent space, so that we have an embedding 
	\begin{equation}(\Lambda^2 M)_0\hookrightarrow \Lambda^2 T_x M.\label{eq:parform}\end{equation}
Therefore $\dim (\Lambda^2 M)_0\leq \frac{n(n-1)}{2}$, and equality in Equation \ref{eq:dim} holds. Hence by a dimension count, the projection $i_V^N\rightarrow (\Lambda^2 M)_0$ is not only injective, but also surjective. 

So we have $\mathfrak{o}(n)\cong (\Lambda^2 M)_0$, and $M$ has the maximal amount of parallel forms it can possibly have (i.e. a space of dimension $\frac{n(n-1)}{2}$). Note that a torus is an example of such a manifold. Motivated by these examples, we claim that $M$ is a flat manifold. 

To prove that $M$ is flat, let us first show that for any $x\in M$, the holonomy group at $x$ is trivial. Recall that the holonomy group consists of linear maps $T_x M\rightarrow T_x M$ obtained by parallel transport along a loop in $M$ based at $x$. Therefore any holonomy map will fix parallel forms pointwise. Suppose now that $T: T_x M\rightarrow T_x M$ is a holonomy map at $x\in M$. We showed above that the evaluation at $x$ is an isomorphism $(\Lambda^2 M)_0\hookrightarrow \Lambda^2 T_x M$ (see Equation \ref{eq:parform}). Since $T$ fixes parallel forms, it is therefore clear that $\Lambda^2 T=\id$ (i.e. $T$ acts trivially on oriented planes in $T_x M$). Since $\dim(M)>2$, it follows that $T=\id$. 

So $M$ has trivial holonomy. Since the holonomy algebra (i.e. the Lie algebra of the holonomy group) contains the Lie algebra generated by curvature operators $R(v,w)$ where $v,w\in T_x M$ (see e.g. \cite[Section 8.4]{petersen}), it follows that $R(v,w)=0$ for all $v,w\in T_x M$, so $M$ is flat.

We will use that $M$ is flat to obtain more information about the Killing fields $i(M)$ of $M$. Recall that the structure of flat manifolds is described by the Bieberbach theorems. Namely, any closed flat manifold is of the form $\bbR^n \slash \Gamma$ for some discrete torsion-free subgroup $\Gamma\subseteq \isom(\bbR^n)$, and there is a finite index normal subgroup $\Lambda\subseteq \Gamma$ that consists of translations of $\bbR^n$ (so $\bbR^n\slash \Lambda$ is a torus). In particular the Killing fields on $\bbR^n\slash \Lambda$ are just obtained by translations of $\bbR^n$, so $i(\bbR^n\slash \Lambda)\cong\bbR^n$ as a Lie algebra.

The Killing fields on $M=\bbR^n\slash\Gamma$ are exactly those Killing fields of $\bbR^n\slash\Lambda$ invariant under the deck group $\Gamma\slash\Lambda$ of the (regular) cover $\bbR^n\slash\Lambda\rightarrow M$. In particular $i(M)$ is a sub-Lie algebra of $\bbR^n$.

Therefore $i(M)$ is abelian. Recall that we have
	$$i(X)\cong i_V^M\oplus (\Lambda^2 M)_0 \oplus i(M)$$
We know that $i_V^N\cong \mathfrak{o}(n)$ has no abelian quotients, so we must have $i_V^N\subseteq i_V^M\oplus (\Lambda^2 M)_0$. Hence for any $x\in N$ and $\widetilde{x}\in \pi_N^{-1}(x)$, we have
	\begin{align*}
	T_{\widetilde{x}} \pi_N^{-1}(x)&=i_V^N|_{\widetilde{x}}\\
	&\subseteq (i_V^M\oplus (\Lambda^2 M)_0)|_{\widetilde{x}}\\
	&\subseteq T_{\widetilde{x}}\pi_M^{-1}(\pi_M(\widetilde{x})),\end{align*}
where the last inclusion holds because the vector fields in $i_V^M\oplus (\Lambda^2 M)_0$ are vertical with respect to $\pi_M$ (see Section \ref{sec:taya}). Since $\pi_N^{-1}(x)$ and $\pi_M^{-1}(\pi_M(\widetilde{x}))$ are connected submanifolds with the same dimension, we must have $\pi_N^{-1}(x)=\pi_M^{-1}(\pi_M(\widetilde{x}))$. Therefore the fibers of $\pi_M$ and $\pi_N$ agree. We conclude that $M$ and $N$ are isometric in the same way as Case 1.
	
%We know $i_V^M\neq i_V^N$, so $i_V^M$ projects nontrivially to $(\Lambda^2 N)_0\rtimes i(N)$. Now consider the image of $(\Lambda^2 N)_0\rtimes i(N)$ in $i_V^M\oplus (\Lambda^2 M)_0\oplus i(M)$. We know $(\Lambda^2 N)_0\rtimes i(N)$ centralizes the image of $i_V^N$ and $i_V^N$ projects nontrivially to $(\Lambda^2 M)_0$. It follows that $(\Lambda^2 N)_0\rtimes i(N)$ projects trivially to $(\Lambda^2 M)_0$, so that 
%	$$(\Lambda^2 N)_0\rtimes i(N)\cong i_V^M\oplus i(M)\cong \mathfrak{o}(n)\oplus \bbR^n.$$
%Therefore $(\Lambda^2 N)_0\cap i_V^M$ is an ideal in $i_V^M\oplus i(M)$, hence $(\Lambda^2 N)_0=i_V^M$ or $(\Lambda^2 N)_0$ projects isomorphically to $\bbR^n$. In the latter case $(\Lambda^2 N)_0=\bbR^n$ as the image of $(\Lambda^2 N)_0$ in $i_V^M$ is an abelian ideal, hence trivial. We claim this is impossible. To see this, note that a parallel form is determined by its values on a single tangent space, hence
%	$$(\Lambda^2 N)_0\subseteq \mathfrak{o}(n).$$
%In particular the rank of $(\Lambda^2 N)_0$ is at most $\rk \mathfrak{o}(n)=\lfloor \frac{n}{2}\rfloor<n$. 
%
%We conclude that $(\Lambda^2 N)_0=i_V^M\cong \mathfrak{o}(n)$ and $i(N)\subseteq i(M)$. The former implies that $N$ is also flat, hence $i(N)\cong \bbR^n$. Hence $\dim i(N)=\dim i(M)$, so $i(N)=i(M)$. 
%
%Now fix a point $x_0\in X$. Define a map
%	$$f: I(N)\rightarrow I(M), g\mapsto h$$
%such that $gx_0 = hx_0$. 
%
%Claim: $f$ is an isometry.
%
%Further $I(N)\rightarrow N, g\mapsto g x$ is a Riemannian cover, and so is $I(M)\rightarrow M$. Then $f$ descends to an isometry $N\rightarrow M$.

\subsection*{Case 3 (many Killing fields)} Assume $i_V^N$ projects nontrivially to $i(M)$. Again we use that $\mathfrak{o}(n)$ is a simple Lie algebra because we have $n>4$. By assumption $i_V^N\cong \mathfrak{o}(n)$ projects nontrivially to $i(M)$, hence $i_V^N$ projects isomorphically to $i(M)$. Let $\frh$ be the image of $i_V^N$ in $i(M)$. At this point we would like to say that $i_V^N\subseteq i(M)$. We cannot in general establish this, but we have the following.
	\begin{claimnr} Assume that $\mathfrak{o}(n)\nsubseteq(\Lambda^2 M)_0$ and that $\mathfrak{o}(n)\nsubseteq (\Lambda^2 N)_0$. Then
	\begin{enumerate}
		\item $i_V^N\subseteq i(M)$, and
		\item $i_V^M\subseteq i(N)$.
	\end{enumerate}
	Therefore $M$ and $N$ have isometry groups of dimension $\geq \frac{1}{2}n(n-1)$.
	\label{cl:swap}
	\end{claimnr}
\begin{proof}Note that $i_V^M$ and $\frh$ centralize each other and are isomorphic to $\mathfrak{o}(n)$. Consider the projection
		$$p_1:\frh\oplus i_V^M\subseteq i(X)\cong i_V^N\oplus (\Lambda^2 N)_0 \oplus i(N)\longrightarrow i_V^N.$$
Note that $\dim(\frh\oplus i_V^M)=2\dim i_V^N$, so $p_1$ cannot be injective. If $p_1$ is trivial, then we have 
	$$\frh\oplus i_V^M\subseteq (\Lambda^2 N)_0\oplus i(N).$$
Using again that $\mathfrak{o}(n)$ is simple, and since $(\Lambda^2 N)_0$ does not contain a copy of $\mathfrak{o}(n)$ by assumption, we must have that $\frh\oplus i_V^M$ projects isomorphically to $i(N)$. However note that $\dim i(N)\leq \frac{n(n+1)}{2}$ by Theorem \ref{thm:maxisom}. Again by comparing dimensions we see that this is impossible. Therefore $\ker p_1$ is a proper ideal of $\frh\oplus i_V^M$, so $\ker p_1$ is either $\frh$ or $i_V^M$. 

Now consider the projection
	$$p_2: \frh\oplus i_V^M\subseteq i(X)\cong i_V^N\oplus (\Lambda^2 N)_0\oplus i(N)\rightarrow i(N).$$
As above we see that $p_2$ can be neither injective nor trivial. Hence we have that $\ker p_2$ is either $\frh$ or $i_V^M$.

If $\ker p_2=i_V^M$, then we have $i_V^M=i_V^N$, but this contradicts the assumption that $i_V^N$ projects nontrivially to $i(M)$. Therefore we must have $\ker p_1=i_V^M$ and $\ker p_2 =\frh$. The latter implies $i_V^N=\frh$, which proves (1).

Since $\ker p_1=i_V^M$, we have $i_V^M\subseteq (\Lambda^2 N)_0\oplus i(N)$ and $i_V^M$ projects trivially to $(\Lambda^2 N)_0$. Therefore we have $i_V^M\subseteq i(N)$, which proves (2).
	\end{proof}
	
If $\mathfrak{o}(n)\subseteq (\Lambda^2 M)_0$ or $\mathfrak{o}(n)\subseteq (\Lambda^2 N)_0$, the proof is finished in Case 2. Therefore we assume $i_V^N\subseteq i(M)$ and $i_V^M\subseteq i(N)$. Write $H_M:=\exp(i_V^N)$ and $H_N:=\exp(i_V^M)$ where $\exp$ is the exponential map on the Lie group Isom$(X)$. Then $H_M$ and $H_N$ are subgroups of Isom$(X)$, each isomorphic to $\SO(n)$, and $M=X\slash H_N$ and $N=X\slash H_M$.

Since $H_M$ and $H_N$ are commuting subgroups of $\isom(X)$, the action of $H_M$ on $X$ descends to an action on $M=X\slash H_N$ with kernel $H_M\cap H_N$. We will write $\overline{H}_M:=H_M\slash (H_M\cap H_N)$ for the group of isometries of $M$ thus obtained. Similarly, $H_N$ acts by isometries on $N=X\slash H_M$ with kernel $H_M\cap H_N$, and we will write $\overline{H}_N:=H_N\slash (H_M\cap H_N)$ for this group of isometries.

Note that $H_M\cap H_N$ is discrete, since its Lie algebra is $i_V^M\cap i_V^N=0$. In particular, since $H_M$ and $H_N$ are compact, it follows that $H_M\cap H_N$ is finite. Therefore the natural quotient map $H_M\rightarrow \overline{H}_M$ is a covering of finite degree, and  $\overline{H}_M$ and $H_M$ have the same Lie algebra. Similarly, $\overline{H}_N$ and $H_N$ have the same Lie algebra. Therefore $\overline{H}_M$ and $\overline{H}_N$ are groups of isometries of closed $n$-manifolds with Lie algebras isomorphic to $\mathfrak{o}(n)$. The results of Section \ref{sec:act} exactly apply to such actions; these results will restrict the possibilities for $M$ and $N$ tremendously, as we will see below. 

Motivated by the results of Section \ref{sec:act}, we will now consider two cases: Either one of $\overline{H}_M$ or $\overline{H}_N$ acts transitively, or neither acts transitively.

\subsection*{Case 3(a) ($\overline{H}_M$ or $\overline{H}_N$ acts transitively)} Suppose $\overline{H}_M$ acts transitively on $M$. Since $\overline{H}_M$ has Lie algebra $\mathfrak{o}(n)$ and $\dim M=n$, Theorem \ref{thm:3case} and Proposition \ref{prop:3case} give a classification of the possibilities for $M$ and $\overline{H}_M$. Since in Cases (1), (2), and (3) of Theorem \ref{thm:3case} the group of isometries is not transitive, but by assumption $\overline{H}_M$ acts transitively on $M$, we know that either
	\begin{itemize}
		\item $M$ is isometric to $S^7\cong \textrm{Spin}(7)\slash G_2$, equipped with a constant curvature metric, and $\overline{H}_M=\textrm{Spin}(7)$, or
		\item $M$ is isometric to $\bbR P^7\cong \SO(7)\slash G_2$, equipped with a constant curvature metric, and $\overline{H}_M=\SO(7)$, or
		\item $M$ is isometric to $\bbC P^3$, equipped with a metric of constant holomorphic sectional curvature, and $\overline{H}_M=\SO(6)\cong \SU(4)\slash\{\pm\id\}$.
	\end{itemize}
We will show that the first case is impossible, and that in the other cases $M$ and $N$ are isometric.

\begin{lemnr} $M$ is not isometric to $S^7$.\end{lemnr}
\begin{proof} Since $\overline{H}_M=H_M\slash (H_M\cap H_N)$, we know that $\overline{H}_M$ is a quotient of $H_M\cong \SO(7)$. In particular, $H_M$ is not simply-connected. On the other hand, $\textrm{Spin}(7)$ is simply-connected. This is a contradiction. \end{proof}

\begin{lemnr} If $M$ is isometric to $\bbR P^7$, then $M$ and $N$ are isometric.\end{lemnr}
\begin{proof} Suppose now $M$ is isometric to $\bbR P^7$, and consider the action of $H_N$ on $N$. From the classification in Theorem \ref{thm:3case} and Remark \ref{rmk:orient}, and using that $\dim(N)=\dim(M)=7$, we see that $N$ must be diffeomorphic to one of the following:
	\begin{enumerate}
		\item $\bbR P^7$,
		\item $S^7$,
		\item $L_N\times  S^1$ where $L_N$ is $S^{6}$ or $\bbR P^{6}$, or
		\item $(S^6\times \bbR)\slash \Gamma$ where $\Gamma\cong D_\infty$ is generated by $(v,t)\mapsto (-v,-t)$ and $(v,t)\mapsto (v,t+2)$.
	\end{enumerate}
\begin{claimnr} We must have that $N$ is diffeomorphic to $\bbR P^7$ (and hence to $M$).\end{claimnr}
\begin{proof} We will show that we can distinguish the frame bundles of the manifolds appearing in Cases (2), (3), and (4) from $\SO(\bbR P^7)$ by their fundamental group.

First, let us compute the fundamental group of $\SO(N)=\SO(\bbR P^7)$. Note that $\SO(S^7)\cong SO(8)$ (see Example \ref{ex:nonpres}). It easily follows that $\SO(\bbR P^7)\cong \SO(8)\slash \{\pm \id\}$. In particular, $\pi_1 \SO(\bbR P^7)$ is obtained as an extension
	$$1\rightarrow \pi_1 \SO(8) \rightarrow \pi_1 \SO(\bbR P^7)\rightarrow \{\pm \id\}\rightarrow 1.$$
So $\pi_1(\SO(\bbR P^7))$ has order 4. So let us now show that in each of the Cases (2), (3), and (4), $\pi_1$ does not have order 4. 
	\begin{itemize}
	\item In Case (2), note that $\pi_1 \SO(S^7)=\pi_1 \SO(8)\cong \bbZ\slash (2\bbZ)$ has order 2. 
	\item In Case (3), $\pi_1 N$ is infinite. By the long exact sequence on homotopy groups for the fiber bundle $\SO(7)\rightarrow \SO(N)\rightarrow N$, we see that $\pi_1 \SO(N)$ surjects onto $\pi_1 N$. Therefore $\pi_1 \SO(N)$ is also infinite.
	\item In Case (4), $\pi_1 N\cong D_\infty$ is infinite as well. The above argument for Case (3) shows that $\pi_1 \SO(N)$ is infinite as well.
	\end{itemize}
The only remaining possibility is that $N$ is diffeomorphic to $\bbR P^7$ (and hence also to $M$). \end{proof}
So we find that $N$ is diffeomorphic to $\bbR P^7$. We will now determine the metric on $N$:
	\begin{claimnr} $N$ has constant curvature.\end{claimnr}
\begin{proof} Theorem \ref{thm:3case} classifies the possible metrics on $N$. Namely, if $\overline{H}_N$ acts transitively on $N$, then $N$ has constant curvature, as desired.

Suppose now that $H_N$ does not act transitively on $N$. Identify the universal cover $\widetilde{N}$ of $N$ (which is diffeomorphic to $S^7$) with the
solution set of $\sum_{i=0}^7 x_i^2=1$ in $\bbR^{8}$. Then by Theorem \ref{thm:3case}.(3) the metric on $\widetilde{N}$ is of the form
	$$ds_{\widetilde{N}}^2=f(x_0) \sum_{i=0}^7 dx_i^2$$
for some smooth positive function $f$ on $[-1,1]$. The function $|x_0|$ descends from $\widetilde{N}$ to $N$, and $\overline{H}_N$ acts isometrically and transitively on each level set
	$$\left\{[x_0,\dots,x_7]\in \bbR P^7 \;\middle|\; \sum_{i=1}^7 x_i^2 = 1 - c^2\right\}$$
for $0\leq c\leq 1$. For $c=0$ this level set is a copy of $\bbR P^6$ (the image of the equator $S^6\subseteq S^7\cong \widetilde{N}$ in $\bbR P^7\cong N$) and for $c=1$ the level set consists of a single point (the image of the north and south pole). For $0<c<1$, the level set is a copy of $S^6$.

Let $x\in N$ be any point with $0<x_0<1$, so that the $\overline{H}_N$-orbit of $x$ is a copy of $S^6$. Since the metric on $\overline{H}_N x$ is given by $f(x_0)\sum_i dx_i^2$, we have
	$$\vol(\overline{H}_N x)=(f(x_0))^{\frac{n}{2}}\vol(S^6)$$ 
where on the right-hand side $\vol(S^6)$ is computed with respect to the standard metric $\sum_i dx_i^2$. Now consider the fiber bundle $\pi_N: \SO(N)\rightarrow N$. Recall that each fiber in $\SO(N)$ has a fixed volume $\nu>0$, and is an $H_M$-orbit. Therefore for $e\in \pi_N^{-1}(x)$, we have
	\begin{equation} \vol(H_M H_N e)=\nu \vol(\overline{H}_N x)=\nu (f(x_0))^{\frac{6}{2}} \vol(S^6).\label{eq:volfiber}\end{equation}
On the other hand, $e$ is a frame at some point $y\in M$. Since the fibers of $\SO(M)\rightarrow M$ also have volume $\nu$, it follows that
	$$\vol(H_M H_N e)=\nu \vol(\overline{H}_M y).$$
Since $\overline{H}_M$ acts transitively on $M$, the right-hand side is just equal to $\nu\vol(M)$. In particular, the left-hand side does not depend on $e$. Using Equation \ref{eq:volfiber}, we see that $f(x_0)$ does not depend on the point $x$ chosen. Since the only requirements for $x$ were that $-1<x_0<1$ and $x_0\neq 0$, we see that $f$ is constant on $(-1,1)\backslash \{0\}$. Since $f$ is also continuous, it is in fact constant on $[-1,1]$, so the metric on $\widetilde{N}$ is given by 
	$$ds_{\widetilde{N}}^2 = c \sum_{i=0}^7 dx_i^2$$
for some $c>0$. Therefore the metric is some multiple of the standard round metric, so $N$ has constant curvature.\end{proof}

So we have shown that both $M$ and $N$ are diffeomorphic to $\bbR P^7$ with constant curvature metrics. Since by Lemma \ref{lem:vol}, we also have that $\vol(M)=\vol(N)$, it follows that $M$ and $N$ have the same curvature, so that they are isometric, as desired. \end{proof}

\begin{lemnr} If $M$ is isometric to $\bbC P^3$, equipped with a metric of constant holomorphic sectional curvature, then $M$ and $N$ are isometric.\end{lemnr}
\begin{proof} Again consider the action of $\overline{H}_N$ on $N$. From the classification in Theorem \ref{thm:3case}, and using that $\dim(N)=\dim(M)=6$, we see that $N$ must be one of the following:
	\begin{enumerate}
		\item diffeomorphic to $S^6$ or $\bbR P^6$, 
		\item diffeomorphic to $L_N\times S^1$ where $L_N$ is $S^{5}$ or $\bbR P^{5}$,
		\item diffeomorphic to $(S^5\times \bbR)\slash \Gamma$ where $\Gamma\cong D_\infty$ is generated by $(v,t)\mapsto (-v,-t)$ and $(v,t)\mapsto (v,t+2)$, or
		\item isometric to $\bbC P^3$ with a metric of constant holomorphic sectional curvature.
	\end{enumerate}
We can rule out Cases (1), (2), and (3) by computations of $\pi_2$. Namely, let us first compute $\pi_2(\SO(\bbC P^3))$. The long exact sequence on homotopy groups of the fibration $\SO(6)\rightarrow \SO(\bbC P^3)\rightarrow \bbC P^3$ gives
	$$1=\pi_2 SO(6)\rightarrow \pi_2(\SO(\bbC P^3))\rightarrow \pi_2(\bbC P^3)\rightarrow \pi_1(SO(6))=\bbZ\slash(2\bbZ).$$
Since $\pi_2(\bbC P^3)\cong \bbZ$ it follows that $\pi_2(\SO(\bbC P^3))\cong \bbZ$. On the other hand, in Case (1), we have $\pi_2(\SO(S^6))=\pi_2(SO(7))=1$ and similarly $\pi_2(\SO(\bbR P^6))=1$. In Case (2), we have that $\pi_2 N \cong \pi_2 L_N$ since $S^1$ is aspherical. Since $L_N$ is diffeomorphic to either $S^5$ or $\bbR P^5$, we have $\pi_2 L_N=1$. Again by the long exact sequence on homotopy groups for the fibration $\SO(N)\rightarrow N$, we see that $\pi_2 \SO(N)=1$. Finally in Case (3) we have $\pi_2 N =\pi_2 S^5=1$. As in Case (2) we have that $\pi_2 \SO(N)=1$.

Therefore in Cases (1), (2), and (3), we cannot have $\SO(N)\cong \SO(\bbC P^3)$, so we conclude that $M$ and $N$ are both isometric to $\bbC P^3$ with a metric of constant holomorphic sectional curvature.

A metric of constant holomorphic sectional curvature on $\bbC P^3$ is determined by a bi-invariant metric on $SU(4)$, which is then induced on the quotient $SU(4)\slash S(U(1)\times U(3))\cong  \bbC P^3$. Hence the metrics on $M$ and $N$ differ only by scaling, so $M$ and $N$ are isometric if and only if $\vol(M)=\vol(N)$. By Lemma \ref{lem:vol} we indeed have $\vol(M)=\vol(N)$ so $M$ and $N$ are isometric.\end{proof}

Above we assumed that $\overline{H}_M$ acts transitively on $M$. If instead $\overline{H}_N$ acts transitively on $N$, the same proof applies verbatim.
%
%There is a classification of compact connected Lie groups acting transitively on spheres as follows \cite{transsphere1,transsphere2}.
%	\begin{enumerate}
%		\item $H^0_x = SO(n)$ for $n-1\geq 2$,
%		\item $H^0_x=SU(\frac{n-1}{2})$ or $U(\frac{n-1}{2})$ for $n-1\geq 4$ even,
%		\item $H^0_x=Sp(\frac{n-1}{4})$ or $Sp(\frac{n-1}{4})\times S^1$ or $Sp(\frac{n-1}{4})\times S^1$ for $n-1\geq 4$ divisible by $4$,
%		\item $H^0_x= G_2$ and $n-1=7$,
%		\item $H^0_x=\textrm{Spin}(7)$ for $n-1=8$,
%		\item $H^0_x=\textrm{Spin}(9)$ for $n-1=16$.
%	\end{enumerate}
%Since we know $\dim H^0_x= \frac{1}{2}(n-1)(n-2)-1$ and the dimension of the sphere is $n-1$, it is easy to see that only (iv) is possible. Then $H$ is a connected Lie group with Lie algebra $o(7)$, so $H=\SO(7)$ or Spin$(7)$. So $M$ is covered by
%	$$\widetilde{M}\cong \textrm{Spin}(7)\slash G_2\cong S^7$$
%(the last diffeomorphism can be seen as follows. Spin($7$) has an eight dimensional spinor representation and $G_2$ is the stabilizer of a vector.) Therefore $M$, being a homogeneous space covered by $S^7$, is diffeomorphic to either $S^7$ or $\bbR P^7$.
\subsection*{Case 3(b) ($H_M$ and $H_N$ do not act transitively)} Theorem \ref{thm:3case} and Proposition \ref{prop:3case} imply that $M$ and $N$ are of one of the following types:
		\begin{enumerate}
			\item diffeomorphic to $S^n$ or $\bbR P^n$ equipped with a metric as in Theorem \ref{thm:3case}.(1),
			\item $L\times S^1$ where each copy $L\times \{z\}$ is an isometrically embedded round sphere or projective space, or
			\item $(S^{n-1}\times\bbR)\slash\Gamma$ where $\Gamma\cong D_\infty$ is generated by $(v,t)\mapsto (v,t+2)$ and $(v,t)\mapsto (-v, -t).$
		\end{enumerate}
\begin{claimnr} $M$ and $N$ belong to the same types in the above classification.\end{claimnr}
\begin{proof} Again we will show that the different types can be distinguished by the fundamental group of the frame bundle. Since $\SO(M)=\SO(N)$, it must then follow that $M$ and $N$ belong to the same type.

The fundamental group of $X=\SO(M)$ can be computed using the long exact sequence on homotopy groups for the fiber bundle $X\rightarrow M$ (or $X\rightarrow N$). Namely, we have
	$$\pi_2(M)\rightarrow \pi_1(SO(n))\rightarrow \pi_1(X)\rightarrow \pi_1(M)\rightarrow 1$$
and likewise for $N$. Since $\pi_2(M)=\pi_2(N)=1$ for all of the above types, we have a short exact sequence
	$$1\rightarrow \bbZ\slash 2\bbZ\rightarrow \pi_1(X)\rightarrow \pi_1(M)\rightarrow 1$$
and likewise for $N$. We see that $\pi_1(X)\cong\bbZ\slash 2\bbZ$ precisely when $M$ is diffeomorphic to $S^n$, and $\pi_1(X)$ has order 4 precisely when $M$ is diffeomorphic to $\bbR P^n$. If $\pi_1(X)$ is infinite then $M$ is of type (2) or (3). If the maximal finite subgroup of $\pi_1(X)$ has order 2 then $M$ is of type (2), and if the maximal finite subgroup of $\pi_1(X)$ has order 4 then $M$ is of type (3). Therefore we can distinguish all the possible cases by considering $\pi_1(X)$, so $M$ and $N$ are of the same type.\end{proof}

We will now show that in each of these cases, $M$ and $N$ are isometric.

\subsection*{Case A ($M$ and $N$ are diffeomorphic to $S^n$ or $\bbR P^n$)} Identify $S^n$ with the solution set of $\sum_{i=0}^{n} x_i^2 = 1$ in $\bbR^{n+1}$. By Theorem \ref{thm:3case}.(1), the metric on $M$ (or its double cover if $M$ is diffeomorphic to $\bbR P^n$) is of the form
	\begin{equation}
	ds_M^2 = f_M(x_0) \sum_{i=0}^n dx_i^2.
	\label{eq:metm}
	\end{equation}
Similarly the metric on $N$ (or its double cover) can be written as
		\begin{equation}
		ds_N^2 = f_N(x_0) \sum_{i=0}^n dx_i^2
		\label{eq:metn}
		\end{equation}
We will now show that $f_M(x)=f_N(x)$ for all $x$. We will just do this in case $M$ and $N$ are diffeomorphic to $S^n$, since the proof for $\bbR P^n$ is similar (note that it is not possible that one of $M$ and $N$ is diffeomorphic to $S^n$, and the other to $\bbR P^n$, since $\SO(S^n)$ and $\SO(\bbR P^n)$ are not diffeomorphic). Theorem \ref{thm:3case}.(1) also describes the action of $\overline{H}_N$ on $N$. Namely, $\overline{H}_N$ leaves the coordinate $x_0$ invariant and acts transitively on each level set of $x_0$. This yields an identification
	$$N\slash \overline{H}_N\cong [-1,1].$$
The $\overline{H}_N$-orbits lying over the points in $(-1,1)$ are copies of $S^{n-1}$, and the orbits lying over $\pm 1$ are fixed points (corresponding to the north and south pole). Similarly we can identify $M\slash \overline{H}_M$ with $[-1,1]$. Of course we can also write $M=X\slash H_M$, and this yields an identification
	$$X\slash(H_M H_N)=M\slash \overline{H}_M.$$
Let $-1<x<1$ and choose a lift $y_M\in M$ of $x$. Equation \ref{eq:metm} shows that $\vol(H_M y_M)=f_M(x)\vol(S^{n-1})$ where $S^{n-1}$ is equipped with the metric $\sum_{i=1}^n dx_i^2$. Similarly if $y_N$ is a lift of $x$ to $N$ we have $\vol(H_N y_N)=f_N(x)\vol(S^{n-1})$. Now choose a common lift $\widetilde{y}$ of $y_M$ and $y_N$ to $X$, i.e. $\widetilde{y}$ is an oriented orthonormal frame at the point $y_M\in M$ and at the point $y_N\in N$. Recall that the volume of a fiber of $X\rightarrow M$ is a fixed constant $\nu>0$. Hence we have
	$$\vol(H_M H_N \widetilde{y})=\nu \vol(\overline{H}_M y_M)=\nu f_M(x) \vol(S^{n-1}).$$
Since the volume of a fiber of $X\rightarrow N$ is also equal to $\nu$, we also have
	$$\vol(H_M H_N \widetilde{y})=\nu \vol(H_N y_N)=\nu f_N(x) \vol(S^{n-1}).$$
It follows that $f_M(x)=f_N(x)$. Hence $M$ and $N$ are isometric.

\subsection*{Case B ($M$ and $N$ are of type (2))} In this case $M$ is diffeomorphic to $L_M\times S^1$ where each copy $L_M\times \{z\}$ of $L_M$ is isometric to a round sphere or projective space. The group $\overline{H}_N$ acts orthogonally on each fiber. However, note that the metric on $M$ is not assumed to be a product metric, but in this case it has to be:
\begin{lemnr} $M$ is isometric to a product $L_M\times S^1$ where $L_M$ is either a round sphere or projective space.\end{lemnr}
\begin{proof} Let $q:M\rightarrow S^1$ be the projection onto the second coordinate. Of course the fibers of $q$ are just the submanifolds $L_M\times \{z\}$ for $z\in S^1$, and form a foliation $\mathcal{L}$ of $M$. Fix an orientation of $L_M$ and define $\SO_\mathcal{L}(M)$ to be the space of pairs $(x,e)$ where $x\in M$ and $e$ is a positively oriented frame for the tangent space at $x$ of the leaf of $\mathcal{L}$ through $x$. There is a natural bundle map $p:\SO_\mathcal{L}(M)\rightarrow M$ defined by $p(x,e):=x$. Further because $\overline{H}_M$ acts isometrically on $M$ preserving the leaves of $\mathcal{L}$, it follows that $\overline{H}_M$ acts on $\SO_\mathcal{L}(M)$.

Of course, explicitly we have $\SO_\mathcal{L}(M)\cong \SO(L_M)\times S^1$, and the bundle map $p:\SO_\mathcal{L}(M)\rightarrow M$ is given by applying the natural bundle map $\SO(L_M)\rightarrow L_M$ to the first coordinate. Next we can explicitly describe the action of $\overline{H}_M$ on $\SO_\mathcal{L}(M)$. Namely, the action of $\overline{H}_M$ on $L_M$ is just the standard action of $\SO(n)$ on $S^{n-1}$ (or the standard action of PSO$(n)$ on $\bbR P^{n-1}$). Using that $\SO(L_M)\cong \SO(n)$ or $\PSO(n)$, we see that $\overline{H}_M$ just acts by left-translations on $\SO(L_M)$. Finally, the action of $\overline{H}_M$ on $\SO_\mathcal{L}(M)\cong \SO(L_M)\times S^1$ is just by left-translations on each copy $\SO(L_M)\times \{z\}$ of $\SO(L_M)$.

The advantage of initially defining $\SO_\mathcal{L}(M)$ more abstractly (in terms of frames for the fibers of $q$), is that we can define an embedding
			$$j: \SO_\mathcal{L}(M)\hookrightarrow \SO(M)$$
in the following way. A point $(x,e)\in \SO_\mathcal{L}(M)$ consists of an oriented orthonormal frame $e$ of the copy of $L_M$ through $x$. Hence $e$ can be extended to a frame for $M$ at $x$ by adding to $e$ the unique unit vector $v\in T_x M$ such  that $(e,v)$ is a positively oriented orthonormal frame for $M$.  We define $j(x,e):=(x,e,v)$. Using that $\overline{H}_M$ preserves each copy $L_M\times \{z\}$ of $L_M$, it is easy to see that $j(\SO_\mathcal{L}(M))$ is an $H_M$-invariant submanifold of $\SO(M)$.  

We equip $\SO_\mathcal{L}(M)$ with the Riemannian metric on $j(\SO_\mathcal{L}(M))$ induced from $\SO(M)$. Since the $H_M$-orbits in $\SO(M)$ are the fibers of the map $\pi_N: X\rightarrow N$, the $H_M$-orbits are totally geodesic in $\SO_\mathcal{L}(M)$ (see Proposition \ref{prop:fibertotgeod}). We conclude that the foliation $\mathcal{F}$ of $\SO_\mathcal{L}(M)$ by $H_M$-orbits is a totally geodesic codimension 1 foliation of $\SO_\mathcal{L}(M)$. Of course this is just the foliation of $\SO_\mathcal{L}(M)=\SO(L_M)\times S^1$ by copies $\SO(L_M)\times \{z\}$ for $z\in S^1$. Consider the horizontal foliation $\mathcal{F}^\perp$ of $\SO_\mathcal{L}(M)$. Since $\mathcal{F}^\perp$ is 1-dimensional, it is integrable.
		
Johnson-Whitt proved that if the horizontal distribution associated to a totally geodesic foliation is integrable, then the horizontal distribution is also totally geodesic \cite[Theorem 1.6]{totgeodfol}. Further they showed that a manifold with two orthogonal totally geodesic foliations is locally a Riemannian product \cite[Proposition 1.3]{totgeodfol}. Therefore $\SO_\mathcal{L}(M)$ is locally a Riemannian product $F\times U$ where $F$ (resp. $U$) is an open neighborhood in a leaf of $\mathcal{F}$ (resp. $\mathcal{F}^\perp$). 
				
Now we show the metric on $M$ has to locally be a product. Recall that the map $p:\SO_\mathcal{L}(M)\rightarrow M$ is defined by $p(x,e)=x$. We have $p=\pi_M\circ j$, where $j:\SO_\mathcal{L}(M)\hookrightarrow \SO(M)$ is the isometric embedding defined above, and $\pi_M:\SO(M)\rightarrow M$ is the natural projection. Since $j$ is an isometric embedding and $\pi_M$ is a Riemannian submersion, it follows that $p$ is also a Riemannian submersion. 

Now let $x\in M$ be any point and choose $\widetilde{x}\in \SO_\mathcal{L}(M)$ with $p(\widetilde{x})=x$. Since the metric on $\SO_\mathcal{L}(M)$ is locally a product, we can choose a neighborhood $\widetilde{U}\times \widetilde{V}$ of $\widetilde{x}$ on which the metric is a product. 

Now let $w=(u,v)\in T_x M\cong T_x \mathcal{L}_x \oplus T_{q(x)} S^1$, where $u\in T_x\mathcal{L}_x$ and $v\in T_{q(x)} S^1$. Let $\widetilde{u}$ (resp. $\widetilde{v}$) be a lift of $u$ (resp. $v$) to $T_{\widetilde{x}} \SO_\mathcal{L}(M)$ that is horizontal with respect to $p$. Set $\widetilde{w}:=(\widetilde{u},\widetilde{v})\in T_{\widetilde{x}} \SO_\mathcal{L} M$, so that $\widetilde{w}$ is a horizontal lift of $w$. Then we have
\begin{align*}
||w||^2 &= ||\widetilde{w}||^2\\
	&=||\widetilde{u}||^2 +||\widetilde{v}||^2\\
	&=||u||^2+||v||^2,
\end{align*}
where on the first and last line we used that $p$ is a Riemannian submersion, and on the second line we used that the metric on $\SO_\mathcal{L}(M)$ is locally a Riemannian product. This shows that the metric on $M$ is locally a product.

It remains to show that the metric on $M$ is globally a product. Recall that $M$ is diffeomorphic to $L_M\times S^1$, and that each copy $L_M\times \{z\}$ (for $z\in S^1$) is isometric to a round sphere or projective space, say with curvature $\kappa(z)$. Therefore to show that the metric is globally a product, it suffices to show that $\kappa$ is constant. This is immediate because the metric on $M$ is locally a product.\end{proof}
	
Of course, the same proof applies to $N$, and shows that $N$ is also isometric to a product $L_N\times S^1$. Further the metrics on the constant curvature spheres or projective spaces $L_M$ and $L_N$ only depend on their curvatures. 
		
\begin{claimnr} $L_M$ and $L_N$ have the same curvature.\end{claimnr}
\begin{proof} Recall that we normalized the Sasaki-Mok-O'Neill metrics on $\SO(M)\cong \SO(N)$ so that the fibers of $\SO(M)\rightarrow M$ and $\SO(N)\rightarrow N$ have volume $\nu$. These fibers are exactly $H_M$ and $H_N$-orbits in $\SO(M)$, and by definition of the Sasaki-Mok-O'Neill metric, the metric restricted to an $H_M$ or $H_N$-orbit is bi-invariant. On the other hand, if we restrict $\pi_M:X\rightarrow M$ to the $H_M$-orbit of a point $x\in X$, we obtain a bundle
\begin{equation}\label{eq:bundle} \pi_M:  H_M x \rightarrow \overline{H}_M \pi_M(x) \cong \mathcal{L}_{\pi_M(x)}.\end{equation}
		 Here $H_M x$ is diffeomorphic to $\SO(n)$ (if the leaves of $\mathcal{L}$ are spheres) or $\text{PSO}(n)$ (if the leaves of $\mathcal{L}$ are projective spaces), and the fiber of the bundle in Equation \ref{eq:bundle} is diffeomorphic to $\SO(n-1)$. 
		 
		 Since the metric on $H_M x$ (viewed as a submanifold of $\SO(M)$) is a bi-invariant metric, the above bundle is isometric to a standard bundle
		 	$$\SO(n)\rightarrow S^{n-1}(r_M) \hspace{1 cm} \text{if } \mathcal{L}_{\pi_M(x)}\cong S^{n-1}$$
		 or 
		 	$$\text{PSO}(n)\rightarrow \bbR P^{n-1}\hspace{1.6 cm} \text{if } \mathcal{L}_{\pi_M(x)}\cong \bbR P^{n-1}$$		 
		 where the base is a round sphere or projective space of some radius $r_M$. It follows that the volume of $H_M x$ only depends on $r_M$. Likewise the volume of $H_N x$ will only depend on the radius $r_N$ of $L_N$. On the other hand we know that $\vol(H_M x)=\vol(H_N x)=\nu$, so we must have that $r_M=r_N$, as desired. 
		 \end{proof}
		 
		 At this point we know that there are $r>0$, $\ell_M>0$ and $\ell_N>0$ such that $M$ is isometric to $S^n(r)\times S^1(\ell_M)$ (or $\bbR P^n(r)\times S^1(\ell_M)$ and $N$ is isometric to $S^n(r)\times S^1(\ell_N)$ (or $\bbR P^n(r)\times S^1(\ell_N)$). It only remains to show that $\ell_M=\ell_N$.
		
		 To see this, we need only recall that by normalization of the Sasaki-Mok-O'Neill metrics, we have $\vol(M)=\vol(N)$ (see Lemma \ref{lem:vol}).

\subsection*{Case C ($M$ and $N$ are of type (3))} The unique torsion-free, index 2 subgroups of $\pi_1(M)$ and $\pi_1(N)$ give double covers $M'$ and $N'$. We claim that the frame bundles $\SO(M')$ and $\SO(N')$ are also isometric. The fiber bundle $\SO(n)\rightarrow X\rightarrow M$ gives
	$$1\rightarrow \bbZ\slash 2\bbZ \rightarrow \pi_1(X)\rightarrow D_\infty\rightarrow 1.$$
Now $\pi_1(\SO(M'))$ and $\pi_1(\SO(N'))$ are both index 2 subgroups of $\pi_1(X)$. Since $M'$ and $N'$ are diffeomorphic to $S^{n-1}\times S^1$ we see that $\pi_1(\SO(M'))\cong (\bbZ\slash 2\bbZ)\times \bbZ$ and likewise for $\pi_1(\SO(N'))$. Therefore $\pi_1(\SO(M'))$ and $\pi_1(\SO(N'))$ correspond to the same index 2 subgroup of $\pi_1(X)$. It follows that $\SO(M')$ and $\SO(N')$ are also isometric.

Since $M'$ and $N'$ are diffeomorphic to $S^{n-1}\times S^1$ and $H_M$ acts on $S^{n-1}$ orthogonally, the argument from Case B applies and yields that $M'$ and $N'$ are isometric to the same product $S^{n-1}\times S^1$. Then $M$ and $N$ are obtained as the quotient of $S^{n-1}\times S^1$ by the map $(v,z)\mapsto (-v,z^{-1})$. Hence $M$ and $N$ are isometric. \end{proof}

\section{Proof for $M$ with positive constant curvature} 
\label{sec:kcnst} 

In the previous section we have proved Theorem \ref{thm:main} in all cases except when $M$ has constant curvature $\frac{1}{2\sqrt{\lambda}}$ or $M$ is a surface. We will resolve the latter case in the next section. In this section we will prove:

\begin{thmnr}Let $M,N$ be closed oriented connected Riemannian $n$-manifolds and assume $M$ has constant curvature $\frac{1}{2\sqrt{\lambda}}$ for some $\lambda>0$. Equip $\SO(M)$ and $\SO(N)$ with Sasaki-Mok-O'Neill metrics using the invariant inner product $\langle \cdot, \cdot\rangle_\lambda$ on $\fro(n)$. Assume $n\neq 2, 3,4,8$. Then $M,N$ are isometric if and only if $\SO(M)$ and $\SO(N)$ are isometric.\label{thm:kcnst}
\end{thmnr}
\begin{proof}
By simultaneously rescaling the metrics on $M$ and $N$ we can assume that the universal cover of $M$ is a round sphere of radius 1. (Note that in the rescaling, we should also rescale the inner product on $\fro(n)$ that is used in the definition of the Sasaki-Mok-O'Neill metric.)

Since $M$ has positive constant curvature, $M$ is a Riemannian quotient of $S^n$ by a finite group of isometries. Since the group of orientation-preserving isometries of $S^n$ is $\SO(n+1)$, we can write $M=S^n\slash \pi_1(M)$ for some (finite) group $\pi_1(M)\subseteq \SO(n+1)$. 

Further we can write $S^n=\SO(n)\backslash \SO(n+1)$ where the quotient is on the left by the standard copy $\SO(n)\subseteq\SO(n+1)$. The action of $\SO(n+1)$ on $S^n$ by isometries is then just the action of $\SO(n+1)$ by right-translations on $\SO(n)\backslash \SO(n+1)$, so that we have 
	$$M\cong \SO(n)\backslash \SO(n+1)\slash \pi_1(M).$$
Passing to the frame bundle, we obtain $X\cong SO(n+1)\slash \pi_1(M)$, where the cover $\SO(n+1)$ is equipped with a bi-invariant metric. Further $N$ is a quotient of $X$ by a group $H_M\cong \SO(n)$ acting effectively and isometrically on $X$.

Consider now the cover $\SO(n+1)\rightarrow X$. The (effective) action of $H_M$ on $X$ lifts to an effective action of a unique connected cover $\hat{H}_M$ of $H_M$ on $\SO(n+1)$. Note that $\SO(n)$ has only one nontrivial connected cover, namely its universal cover $\text{Spin}(n)$. Therefore we have either $\hat{H}_M\cong \SO(n)$ or $\hat{H}_M\cong \text{Spin}(n)$. We can actually describe the action of $\hat{H}_M$ on $\SO(n+1)$ precisely:

\begin{claimnr} $\hat{H}_M$ is isomorphic to $\SO(n)$ and acts on $\SO(n+1)$ by either left- or right-translations.\end{claimnr}
\begin{proof} Consider the full isometry group of $\SO(n+1)$ (with respect to a bi-invariant metric), which has been computed by d'Atri-Ziller \cite{isomliegp}. Namely, they show that the isometry group of a simple compact Lie group $G$ equipped with a bi-invariant metric is
	$$\isom(G)\cong G\rtimes \Aut(G)$$
where the copy of $G$ acts by left-translations on $G$. We apply this to the group $G=\SO(n+1)$. Since $\hat{H}_M$ is connected, it follows that the image of $\hat{H}_M\hookrightarrow \isom(G)$ is contained in the connected component $\isom(G)^0$ of $\isom(G)$ containing the identity. We can explicitly compute $\isom(G)^0$. Namely, since $\Out(G)$ is discrete, $\isom(G)^0$ is isomorphic to 
	$$G\rtimes \Inn(G)\cong (G\times G)\slash Z(G)$$
where $Z(G)$ is the center of $G$, and $Z(G)\hookrightarrow G\times G$ is the diagonal embedding. The two copies of $G$ act by left- and right-translations on $G$. 

It will be convenient to work with the product $G\times G$, rather than $(G\times G)\slash Z(G)$. Note that the preimage of $\hat{H}_M$ under the natural projection
	$$G\times G\rightarrow (G\times G)\slash Z(G)$$
is a (possibly disconnected) cover of $\hat{H}_M$. Let $\widetilde{H}_M$ denote the connected component containing the identity (so $\widetilde{H}_M$ is a connected cover of $H_M$, and hence isomorphic to either $\SO(n)$ or $\text{Spin}(n)$). 

We will first show that $\widetilde{H}_M$ has to be contained in a single factor of $G\times G$. To see this, let $p_i:\widetilde{H}_M\rightarrow G$ be the projection to the $i$th factor (where $i=1,2$). Since $\widetilde{H}_M$ is a simple connected Lie group, $p_i$ either has finite kernel or is trivial. 

Further at least one of the projections has to be faithful: First, if one of the projections is trivial, then $\widetilde{H}_M$ is contained in a single factor, so that the other projection is faithful. Therefore to show one of the projections has to be faithful, it suffices to consider the case where neither projection is trivial, so that both projections have finite kernel. Let $K_i, i=1,2$ be the kernels of the projections of $\widetilde{H}_M$ onto the $i$th factor. Then $K_i$ is a discrete normal subgroup of $\widetilde{H}_M$, and hence central. As discussed above, the only possibilities for $\widetilde{H}_M$ are $\SO(n)$ and $\text{Spin}(n)$. The center $Z(\widetilde{H}_M)$ of $\widetilde{H}_M$ is then
	$$Z(\widetilde{H}_M)\cong \begin{cases} 
	1 					& \text{if }\widetilde{H}_M\cong \SO(n), n \text{ is odd},\\ 
	\bbZ\slash(2\bbZ) 	& \text{if }\widetilde{H}_M\cong \SO(n), n \text{ is even},\\ 
	\bbZ\slash(2\bbZ) 	& \text{if }\widetilde{H}_M\cong \text{Spin}(n), n \text{ is odd},\\
	\bbZ\slash(4\bbZ) & \text{if }\widetilde{H}_M\cong \text{Spin}(n), n \text{ is even}.
	\end{cases}$$
Further, since no nontrivial element of $\widetilde{H}_M$ projects trivially to both factors (for such an element would be trivial in $G\times G$), we must have $K_1\cap K_2=1$. On the other hand, none of the possibilities for $Z(\widetilde{H}_M)$ have two nontrivial subgroups that intersect trivially, so we conclude that $K_1$ or $K_2$ is trivial. Without loss of generality, we assume that $K_1=1$.

Therefore to prove the claim that $\widetilde{H}_M$ is contained in a single factor, we must show that $p_2(\widetilde{H}_M)$ is trivial. Suppose it is not. Then $p_2$ has finite kernel, so $p_2(\widetilde{H}_M)$ is a subgroup of $G=\SO(n+1)$ of dimension $\dim \widetilde{H}_M=\frac{n(n-1)}{2}$. Fortunately, there are very few possibilities by the following fact:
\begin{lemnr}[{\cite[Lemma 1 in II.3]{kotrgps}}] Let $H$ be a closed connected subgroup of $\SO(n+1)$ of dimension $\frac{n(n-1)}{2}$ with $n+1\neq 4$. Then either
		\begin{enumerate}
			\item $H\cong \SO(n)$ and $H$ fixes a line in $\bbR^{n+1}$, or
			\item $H\cong \textrm{Spin}(7)$ (and hence $n+1=8$), and $H$ is embedded in $\SO(8)$ via a spin representation.
		\end{enumerate}
		\label{lem:poss}
	\end{lemnr}
Here we say that a representation of Spin$(n)$ is \emph{spin} if it does not factor through the covering map $\text{Spin}(n)\rightarrow \SO(n)$. To obtain the desired contradiction, we will now consider various cases depending on which of the above possibilities describe $p_1(\widetilde{H}_M)$ and $p_2(\widetilde{H}_M)$. For ease of notation we set $\widetilde{H}_i:=p_i(\widetilde{H}_M)$ for $i=1,2$. Before considering each case separately, let us first make the following basic observation that underlies the argument in each case:

Recall that $H_M$ acts freely on $X$. It follows that that $\widetilde{H}_M\slash (Z(G)\cap \widetilde{H}_M)$ acts freely on $G$: Namely, if $h\in \widetilde{H}_M$ fixes $x\in G$, then the image of $h$ under $\widetilde{H}_M\rightarrow H_M$ fixes the image of $x$ under the covering map $G\rightarrow X$. Since $H_M$ acts freely on $X$, we see that $h$ belongs to the kernel of $\widetilde{H}_M\rightarrow X$. Since the map $G\rightarrow X$ is equivariant with respect to the morphism $\widetilde{H}_M\rightarrow H_M$, it follows that for any $g\in G$, the points $g$ and $h\cdot g$ of $G$ have the same image in $X$. This exactly means that the action of $h$ on $G$ is a deck transformation of the covering $G\rightarrow X$. Since $h$ fixes the point $x\in G$ and any deck transformation that fixes a point is trivial, $h$ acts trivially on $G$. Since the kernel of the action of $G\times G$ on $G$ is the center $Z(G)$, it follows that $h$ is central, as desired.

Therefore if $h=(h_1,h_2)\in \widetilde{H}_M\subseteq G\times G$ fixes a point in $G$, then $h_1=h_2$ and $h_i$ are central in $G$. Since $(h_1,h_2)\cdot g = h_1 g h_2^{-1}$, the stabilizer of $g\in G$ consists exactly of the elements of the form $(h_1, g h_1 g^{-1})$ where $h_1\in G$. Our strategy for obtaining a contradiction in each of the cases below is to find an element $h=(h_1,g h_1 g^{-1})\in \widetilde{H}_M$ but with $h_1\notin Z(G)$.

\subsection*{Case 1 ($\widetilde{H}_1$ and $\widetilde{H}_2$ are both of Type (1) of Lemma \ref{lem:poss})} By assumption, there are nonzero vectors $v_1$ and $v_2\in \bbR^{n+1}$ such that $\widetilde{H}_i\cong \SO(n)$ fixes $v_i$. The representation of $\widetilde{H}_i$ on $(\bbR v_i)^\perp$ is the standard representation of $\SO(n)$. Therefore there is some an intertwiner $T:(\bbR v_1)^\perp \rightarrow (\bbR v_2)^\perp$ of these representations. Recall that an irreducible representation leaves invariant at most one inner product up to positive scalars (for if $Q_1$ and $Q_2$ are linearly independent invariant bilinear forms, then a suitable linear combination $Q=\alpha Q_1+\beta Q_2$ is invariant and degenerate as a bilinear form; the kernel of $Q$ is then a proper invariant subspace). It follows that after possibly replacing $T$ by $\lambda T$ for some $\lambda>0$, the intertwiner $T$ is orthogonal.

We can extend $T$ to an intertwiner $\bbR^{n+1}\rightarrow \bbR^{n+1}$ between $\widetilde{H}_1$ and $\widetilde{H}_2$ by setting $T v_1 := \mu v_2$ for some $\mu \neq 0$. We will denote the extension by $T$ as well. By choosing $\mu$ suitably, we can arrange that $T$ is orthogonal, and after possibly changing the sign of $\mu$, we can also arrange that $\det T=1$. 

The map $T$ then belongs to $\SO(n+1)$, so that we have that 
	$$\widetilde{H}_M=\{(h, T h T^{-1})\mid h\in \widetilde{H}_1\}.$$
As observed above, it follows that $\widetilde{H}_M$ does not act freely on $X$.

\subsection*{Case 2 (At least one of $\widetilde{H}_1$ and $\widetilde{H}_2$ is of Type (2) of Lemma \ref{lem:poss})} Note that it is not possible that $\widetilde{H}_1$ is of Type (1) and $\widetilde{H}_2$ is of Type (2). Namely, in this case we would have that $\widetilde{H}_M\cong \SO(n)$ (because $\widetilde{H}_M\cong \widetilde{H}_1$), but the map $\widetilde{H}_M\rightarrow \widetilde{H}_2$ would be a covering $\SO(n)\rightarrow \text{Spin}(n)$, which is impossible (since the latter is simply-connected but the former is not).

So we must have that $\widetilde{H}_1$ is of Type (2). In particular we have $n=7$. Unfortunately, we cannot immediately apply the same argument as in Case 1, because Spin$(7)$ has multiple faithful representations of dimension 8. This difficulty is resolved by passing to a suitable subgroup of Spin$(7)$: Namely given a spin representation of Spin$(7)$, the stabilizer of any nonzero $v\in \bbR^8$ is isomorphic to the exceptional simple Lie group $G_2$.

For the rest of the proof we fix some nonzero $v\in \bbR^8$ and let $L$ be the stabilizer in $\widetilde{H}_1$ of $v$. We have two representations of $L$ on $\bbR^8$: On the one hand we have $L\subseteq \widetilde{H}_1$. On the other hand we can consider $p_2(p_1^{-1}(L))\subseteq \widetilde{H}_2$. We analyze these representations in turn and will show they are equivalent. Before doing so, it will be helpful to recall some classical facts about the representation theory of $G_2$ (see \cite[Chapter 5]{g2ref}) for (a)-(d) and \cite[Table X.6.IV]{helgason} for (e)):
\begin{enumerate}[(a)]
	\item $G_2$ is obtained as the subgroup of matrices of $\SO(8)$ that preserve the product of the octonions $\mathbb{O}$,
	\item $G_2$ has no nontrivial representations of dimension less than 7,
	\item $G_2$ has a single representation of dimension 7 (the action on the purely imaginary octonions) that by Fact (b) is necessarily irreducible,
	\item $G_2$ has no irreducible representation of dimension 8, and
	\item $G_2$ has trivial center.
\end{enumerate}
We will write $\mathds{1}$ for the trivial representation and $\text{Im}(\mathbb{O})$ for the unique 7-dimensional faithful representation.

Let us now consider the first representation, obtained by considering $L$ as a subgroup of $\widetilde{H}_1$. This representation is automatically faithful and has $\bbR v$ as a trivial summand. The summand $(\bbR v)^\perp$ is therefore a faithful 7-dimensional representation and by Fact (c) equivalent to $\text{Im}(\mathbb{O})$. Therefore the first representation is equivalent to $\mathds{1}\oplus \text{Im } \mathbb{O}$. 

We turn to the second representation, obtained by the map $p_2\circ p_1^{-1}:L\rightarrow \widetilde{H}_2$. This is a map with finite kernel (because $p_2$ has finite kernel and $p_1$ is an isomorphism), so that the kernel is contained in the center. Since $G_2$ has no center (see Fact (e)), it follows that this representation is also faithful. Since $G_2$ has no irreducible representation of dimension 8  (see Fact (d)), we must have that the second representation also decomposes as $\mathds{1}\oplus\text{Im } \mathbb{O}$. Therefore there is an intertwiner $T:\bbR^8\rightarrow \bbR^8$ between these representations. The rest of the argument proceeds exactly as in Case 1.

This concludes the proof that $\widetilde{H}_M$ is contained in one of the factors of $G\times G$. To complete the proof of the claim, we must show that $\widetilde{H}_M\cong \SO(n)$. By the dichotomy from Lemma \ref{lem:poss}, the only other possibility is that $n=7$ and $\widetilde{H}_M$ is given by a spin representation of Spin$(7)$. 

In the latter case, we can see that $N$ has constant positive curvature: Namely, since $M$ has constant curvature, the metric on $X\cong \SO(8)\slash \pi_1(M)$ lifts to a bi-invariant metric on $\SO(8)$ and hence to a bi-invariant metric on Spin$(8)$. On the other hand $N=X\slash H_M$ is finitely covered by $\SO(8)\slash \text{Spin}(7)$, and hence also by Spin$(8)\slash \text{Spin}(7)$. It is well-known that a bi-invariant metric on Spin(8) induces a metric of constant positive curvature on $S^7\cong \text{Spin}(8)\slash\text{Spin}(7)$.

Since $N$ has constant positive curvature, we can write $N=\SO(7)\backslash \SO(8)\slash \pi_1(N)$ for some finite subgroup $\pi_1(N)\subseteq \SO(8)$ acting by right-translations. The frame bundle of $N$ is then $X=\SO(8)\slash \pi_1(N)$ with $H_M\cong \SO(7)$ acting by left-translations. This contradicts that $\widetilde{H}_M$ was given by a spin representation into $\SO(8)$, and hence finishes the proof of the claim.\end{proof}

Since $H_M$ acts by left or right-translations on $\SO(n+1)$, we will identify $H_M$ with a subgroup of $\SO(n+1)$. Then we can conjugate $H_M$ to a standard copy of $\SO(n)$ by an element of $\SO(n+1)$. Therefore without loss of generality we have $N\cong SO(n)\backslash SO(n+1)\slash \pi_1(N)$, and we have an isometry
	$$f: \SO(n+1)\slash \pi_1(M)\cong SO(M) \rightarrow SO(N)\cong \SO(n+1)\slash \pi_1(N).$$
By composing with a left-translation of $\SO(n+1)$, we can assume $f(e \, \pi_1(M))=e \, \pi_1(N)$. It remains to show there is an isometry
	$$M\cong \SO(n)\backslash \SO(n+1)\slash \pi_1(M)\rightarrow \SO(n)\backslash \SO(n+1)\slash \pi_1(N)\cong N.$$
	\begin{claimnr} $f$ lifts to an isometry $\SO(n+1)\rightarrow \SO(n+1)$.\label{cl:lift}\end{claimnr}
	\begin{proof} The universal cover of $\SO(M)$ and $\SO(N)$ is Spin$(n+1)$, so $f$ lifts to a map
		$$\widetilde{f}:\text{Spin}(n+1)\rightarrow \text{Spin}(n+1).$$
	We can choose the lift $\widetilde{f}$ such that $\widetilde{f}(e)=e$, where $e$ is the identity element of $\SO(n+1)$. Note that since $f$ is an isometry, $\widetilde{f}$ is an isometry as well (with respect to a bi-invariant metric on Spin$(n+1)$). As previously mentioned, d'Atri-Ziller computed the group of isometries of a connected compact semisimple Lie group $G$ \cite{isomliegp}. Indeed, Isom$(G)=G\rtimes \Aut(G)$, where the copy of $G$ acts by left-translations. It immediately follows that any isometry fixing the identity element $e$ is an automorphism. Therefore $\widetilde{f}$ is an automorphism of $\text{Spin}(n+1)$. 
	
	Recall that Spin$(n+1)$ has a unique central element $z$ of order 2, and we have $\SO(n+1)=\text{Spin}(n+1)\slash \langle z\rangle$. Since $z$ is the unique central element of order 2, we must have that $\widetilde{f}(z)=z$. It follows that $\widetilde{f}$ descends to an automorphism of $\SO(n+1)$, as desired.\end{proof}
	
	 Let
 	$$\hat{f}:\SO(n+1)\rightarrow \SO(n+1)$$
 denote a lift of $f$. As above, by choosing an appropriate lift, we can assume that $\hat{f}(e)=e$, and hence that $\hat{f}$ is an automorphism of $\SO(n+1)$ (here we again used the computation of d'Atri-Ziller of the isometry group of $\SO(n+1)$). Because $\hat{f}$ is a lift of $f$, we know that $\hat{f}$ restricts to an isomorphism $\pi_1 M\rightarrow \pi_1 N$. 
 
 Since $\hat{f}$ is an automorphism of $\SO(n+1)$, there is some $g\in \SO(n+1)$ such that $\hat{f}(\SO(n))=g \SO(n) g^{-1}$. Here, as well as well as below, we identify $\SO(n)$ with a fixed standard copy in $\SO(n+1)$. Define a map
 	$$\hat{\varphi}:\SO(n+1)\rightarrow \SO(n+1)$$
by $\hat{\varphi}(x):=g^{-1} \hat{f}(x)$. 

\begin{claimnr} We have
	\begin{enumerate}
	\item $\hat{\varphi}$ is an isometry, 
	\item for any $x\in \SO(n+1)$, we have $\hat{\varphi}(\SO(n) x)=\SO(n) \hat{\varphi}(x)$, and 
	\item for any $x\in \SO(n+1)$, we have $\hat{\varphi}(x\, \pi_1(M)) = \hat{\varphi}(x)\pi_1(N)$.
	\end{enumerate}
\label{cl:equivar}
\end{claimnr}
\begin{proof}[Proof of (1)] Since left-translation by $g$ is an isometry of $\SO(n+1)$, and $\hat{f}$ is also an isometry of $\SO(n+1)$, it follows that the map $\hat{\varphi}$ is an isometry. 
\subsubsection*{Proof of (2)} Let $x\in \SO(n+1)$. We have $\hat{\varphi}(\SO(n) x)= g^{-1} \hat{f}(\SO(n) x)$. Since $\hat{f}$ is an automorphism of $\SO(n+1)$, we then have
	$$\hat{\varphi}(\SO(n) x) = g^{-1} \hat{f}(\SO(n)) \hat{f}(x).$$
Using that $\hat{f}(\SO(n))=g \SO(n) g^{-1}$, we see that
	$$\hat{\varphi}(\SO(n) x)= \SO(n) g^{-1}\hat{f}(x)= \SO(n) \hat{\varphi}(x).$$
\subsubsection*{Proof of (3)} Let $x\in\SO(n+1)$. This is similar to the proof of (2), but now using that $\hat{f}(\pi_1(M))=\pi_1(N)$. We have
	\begin{align*}
	\hat{\varphi}(x \pi_1(M))	&=g^{-1}\hat{f}(x \, \pi_1(M))\\
					&=g^{-1}\hat{f}(x)\hat{f}(\pi_1(M))\\
					&=\hat{\varphi}(x) \pi_1(N).
	\end{align*}
\end{proof}
From Properties (2) and (3) of Claim \ref{cl:equivar}, it is immediate that $\hat{\varphi}$ descends to a map 
	$$\overline{\varphi}:\SO(n)\backslash \SO(n+1)\slash \pi_1(M)\rightarrow \SO(n) \backslash \SO(n+1)\slash \pi_1(N).$$
\begin{claimnr} $\overline{\varphi}$ is an isometry $M\rightarrow N$.\end{claimnr}
\begin{proof} Recall that at the end of Case 1 of the proof of Theorem \ref{thm:mainmost}, we showed that an isometry $X\rightarrow X$ that maps the fibers of $\pi_M:X\rightarrow M$ to the fibers of $\pi_N:X\rightarrow N$, descends to an isometry $M\rightarrow N$. 

In the current setting, the map $\hat{\varphi}:\SO(n+1)\rightarrow \SO(n+1)$ descends to a map
	$$\varphi:X\cong \SO(n+1)\slash \pi_1(M)\rightarrow \SO(n+1)\slash\pi_1(N)\cong X$$
by Property (3) of Claim \ref{cl:equivar}. Since $\hat{\varphi}$ is an isometry and the maps 
	$$\SO(n+1)\rightarrow \SO(n+1)\slash \pi_1(M) \hspace{1 cm} \text{and} \hspace{1 cm} \SO(n+1)\rightarrow \SO(n+1)\slash \pi_1(N)$$
are Riemannian coverings, it follows that $\varphi$ is an isometry.

Therefore to prove the claim, it suffices to show that $\varphi$ maps fibers of $X\rightarrow M$ to fibers of $X\rightarrow N$. Under the identifications $X\cong \SO(n+1)\slash \pi_1(M)$ and $M\cong \SO(n)\backslash \SO(n+1)\slash \pi_1(M)$, the map $X\rightarrow M$ is just the natural orbit map
	$$\SO(n+1)\slash \pi_1(M) \rightarrow \SO(n)\backslash \SO(n+1)\slash \pi_1(M).$$
Therefore the fibers of $X\rightarrow M$ are exactly the $\SO(n)$-orbits in $\SO(n+1)\slash \pi_1(M)$ (under the action by left-translation). Likewise, the fibers of $X\rightarrow N$ are the $\SO(n)$-orbits in $\SO(n+1)\slash \pi_1(N)$ under the action by left-translation. It follows immediately from Property (3) of Claim \ref{cl:equivar} that $\hat{\varphi}$ maps $\SO(n)$-orbits to $\SO(n)$-orbits, and hence so does $\varphi$. \end{proof}

We have shown that the map $\overline{\varphi}$ is an isometry $M\rightarrow N$, so that $M$ and $N$ are isometric, which finishes the proof of Theorem \ref{thm:kcnst}. \end{proof}

\section{Proof of the main theorem for surfaces}
\label{sec:surf}

In this section we prove Theorem \ref{thm:main} for surfaces. We cannot use the Takagi-Yawata theorem (Theorem \ref{thm:tyexist}) that computes $i(X)$ in this situation, but instead we use the classification of surfaces and Lie groups in low dimensions.

Let $M$ and $N$ be closed oriented surfaces with $\SO(M)\cong SO(N)$. Therefore $M$ and $N$ are each diffeomorphic to one of $S^2$, $T^2$ or $\Sigma_g$ with $g\geq 2$. We know that
	\begin{itemize}
		\item $\SO(S^2)$ is diffeomorphic to $\SO(3)$,
		\item $\SO(T^2)$ is diffeomorphic to $T^3$, and
		\item $\SO(\Sigma_g)$ is diffeomorphic to $T^1 \Sigma_g=\PSL_2\bbR\slash \Gamma$ for a cocompact torsion-free lattice $\Gamma\subseteq \PSL_2\bbR$.
	\end{itemize}
In particular the diffeomorphism type of the frame bundle of a surface determines the diffeomorphim type of the surface. It follows that $M$ and $N$ are diffeomorphic.
	
Consider the Lie algebra of Killing fields $i(X)$ of $X$. Then $i(X)$ contains the (1-dimensional) subalgebras $i_V^M$ and $i_V^N$. If $i_V^M=i_V^N$, then we proceed as in Case 1 in Section \ref{sec:main}, and we find that $M$ and $N$ are isometric. Therefore we will assume that $i_V^M\neq i_V^N$. In particular we must have $\dim i(X)\geq 2$.

As before, let $H_M$ (resp. $H_N$) be the subgroup of $\isom(X)$ obtained by exponentiating the Lie algebra $i_V^N$ (resp. $i_V^M$). Then $H_M$ and $H_N$ are closed subgroups of $\isom(X)$ isomorphic to $S^1$.

We will now consider each of the possibilities of the diffeomorphism types of $M$ and $N$, and prove that $M$ and $N$ have to be isometric. 

\subsection*{Case 1 ($M$ and $N$ are diffeomorphic to $\Sigma_g,~ g\geq 2$)} Then $X=T^1\Sigma_g$ is a closed aspherical manifold. Conner and Raymond proved \cite{actasphmnfd} that if a compact connected Lie group $G$ acts effectively on a closed aspherical manifold $L$, then $G$ is a torus and $\dim G\leq \rk_\bbZ Z(\pi_1 L)$, where $Z(\pi_1 L)$ is the center of $\pi_1(L)$. In particular we find that $\dim i(X)\leq \rk_\bbZ Z(\pi_1 T^1\Sigma_g)=1$. This contradicts our assumption that $\dim i(X)\geq 2$.

\subsection*{Case 2 ($M$ and $N$ are diffeomorphic to $S^2$)} Let $G$ be the connected component of $\isom(X)$ containing the identity. Then $G$ is a compact connected Lie group acting effectively and isometrically on $X=\SO(3)$, and $G$ contains $H_M$ and $H_N$.

If $\dim G=2$, then $G$ is a 2-torus. In particular $H_M$ and $H_N$ centralize each other. Therefore $H_N$ acts on $X\slash H_M=N$ and similarly $H_M$ acts on $M$. The kernel of either of these actions is $H_M\cap H_N$, which is a finite subgroup of both $H_M$ and $H_N$. 

Since an $S^1$-action on $S^2$ has at least one fixed point (because $\chi(S^2)\neq 0$), we see that $N\slash H_N \cong [-1,1]\cong M\slash H_M$. It is then straightforward to see that the metric on $M$ (resp. $N$) is of the form
	$$ds_M^2 = f_M(x_0)(dx_0^2 + dx_1^2+dx_2^2)$$
(resp. $ds_N^2 =f_N(x_0)(dx_0^2 + dx_1^2+dx_2^2)$) as in Theorem \ref{thm:3case}.(1). We can apply the reasoning from Case A of the proof of Case 3(b) in Section \ref{sec:main} to show $M$ and $N$ are isometric.

Therefore we will assume $\dim G\geq 3$. In addition we know that $\dim G\leq 6$ by Theorem \ref{thm:maxisom}. Finally, we must have $\text{rank}(G)\leq 2$: Namely let $T$ be a maximal torus in $G$ containing $H_N$. Since $T$ centralizes $H_N$, the group $T\slash H_N$ acts effectively on $M$. However, a torus of dimension $\geq 2$ does not act effectively on $S^2$. (To see this, note that any 1-parameter subgroup $H$ has a fixed point on $S^2$ because the Killing field generated by $H$ has a zero on $S^2$. We can take $H$ to be dense, so that the entire torus fixes a point $p$. The isotropy action on $T_p M$ is a faitful 2-dimensional representation of the torus, which is impossible unless the torus is 1-dimensional.)

Therefore the only possibilities for the Lie algebra $\frg$ of $G$ are 
	\begin{enumerate}[(a)]
		\item $\frg\cong \frak{o}(3)$,
		\item $\frg\cong \bbR\oplus \frak{o}(3)$, and
		\item $\frg\cong \frak{o}(3)\oplus \frak{o}(3)$.
	\end{enumerate}
We will now consider each of these cases separately.
\subsection*{Case 2(a) ($\frg\cong \frak{o}(3)$)} Since $G$ has rank 1, $H_M$ and $H_N$ are both maximal tori of $G$. Since all maximal tori are conjugate, there is some element $g\in G$ so that $g H_N g^{-1}=H_M$. Then $g$ induces an obvious isometry $M\rightarrow N$.

\subsection*{Case 2(b) ($\frg\cong \bbR\oplus \frak{o}(3)$)} We can conjugate $H_N$ by an element $g\in G$ so that $g H_N g^{-1}$ and $H_M$ centralize each other. Then either $g H_N g^{-1}=H_M$, in which case $g$ induces an isometry $M\rightarrow N$, or $g H_N g^{-1}$ and $H_M$ generate a 2-torus. In the latter case the argument above in case $\dim G=2$ shows that the metrics on $M$ and $N$ are of the form
	$$ds^2 = f(x_0)(dx_0^2 + dx_1^2 + dx_2^2)$$
for some function $f$ on $[-1,1]$. Then the argument of Case A of Case 3(b) in Section \ref{sec:main} shows that $M$ and $N$ are isometric.
\subsection*{Case 2(c) ($\frg\cong \frak{o}(3)\oplus \frak{o}(3)$)} In this case $\dim \isom(X)=6$ is maximal. By Theorem \ref{thm:maxisom} the metric on $X$ has positive constant curvature. Therefore the metrics on $M$ and $N$ have positive constant curvature. Further by Lemma \ref{lem:vol} we have $\vol(M)=\vol(N)$. It follows that $M$ and $N$ are isometric.

\subsection*{Case 3 ($M$ and $N$ are diffeomorphic to $T^2$)} In this case $X$ is diffeomorphic to $T^3$. Again by the theorem of Conner-Raymond \cite{actasphmnfd} on actions of compact Lie groups on aspherical manifolds, we know that a connected compact Lie group acting effectively on a torus is a torus. Therefore $H_N$ and $H_M$ centralize each other, so $H_M$ and $H_N$ generate a 2-torus. Further $H_M$ acts on $M=X\slash H_N$ with finite kernel $H_M\cap H_N$. Again by \cite{actasphmnfd}, the action of $H_M\slash (H_M\cap H_N)$ on $M$ is free, so that the map
	$$M\rightarrow M\slash H_M\cong S^1$$
is a fiber bundle (with $S^1$ fibers). The argument of Case B in Case 3(b) of the proof of Theorem \ref{thm:main} constructs a (unit length) Killing field $X_M$ on $M$ that is orthogonal to the fibers of $M\rightarrow M\slash H_M$. It follows that $M$ is a 2-torus equipped with a translation-invariant metric. Any such metric is automatically flat: Namely, because the torus is abelian, the metric is automatically bi-invariant. Then we use the following general fact: On a Lie group $H$ with a bi-invariant metric, the Lie structure and sectional curvature are tied by the identity (see e.g. \cite[Proposition 3.4.12]{petersen})
$$K(X,Y)=\frac{1}{4}||[X,Y]||^2$$
where $X,Y$ are orthonormal vectors in $\frh$ (which is identified with $T_e H$ in the usual way), and the bracket is the Lie bracket. Since $T^2$ is abelian, it follows that the sectional curvatures with respect to any invariant metric vanish.

We conclude that $M$ is flat. By carrying out the same construction for $N$, we obtain a Killing field $X_N$ on $N$ that is orthogonal to the $H_N$-orbits, and we conclude that $N$ is flat.

To show that $M$ and $N$ are isometric, recall that the isometry type of a flat 2-torus is specified by the length of two orthogonal curves that generate its fundamental group. For $M$ we can consider the curves given by an $H_M$-orbit on $M$ and an integral curve of $X_M$. Similarly for $N$ we can consider an $H_N$-orbit on $N$ and an integral curve of $X_N$.

For $x\in M$ and $\widetilde{x}\in X$ lying over $x$, we have a covering 
	$$H_M \widetilde{x}\rightarrow H_M x$$
of degree $|H_N\cap H_M|$. Recall that the $H_M$-orbits in $X$ have a fixed volume $\nu$, since we normalized the Sasaki-Mok-O'Neill metric on $X$ in this way. Therefore
	$$\ell(H_M x)=\frac{1}{|H_N\cap H_M|}\ell(H_M\widetilde{x})=\frac{\nu}{|H_N\cap H_M|}.$$
Combining this with a similar computation for the length of an $H_N$-orbit on $N$ gives $\ell(H_M x)=\ell(H_N y)$ for every $x\in M$ and $y\in N$. Therefore we see that the length of an integral curve of $X_M$ (resp. $X_N$) is $\frac{\vol(M)}{\ell(H_M\cdot x)}$ for $x\in M$ (resp. $\frac{\vol(N)}{\ell(H_N \cdot y)}$ for $y\in N$). Since $\vol(M)=\vol(N)$ by Lemma \ref{lem:vol}, it follows that $M$ and $N$ are isometric.

\bibliographystyle{alpha}
\bibliography{isomframebib}

\begin{thebibliography}{Wan47}

\bibitem[Ada96]{g2ref}
J.~F. Adams.
\newblock {\em Lectures on exceptional {L}ie groups}.
\newblock Chicago Lectures in Mathematics. University of Chicago Press,
  Chicago, IL, 1996.

\bibitem[CR70]{actasphmnfd}
P.~E. Conner and Frank Raymond.
\newblock Actions of compact {L}ie groups on aspherical manifolds.
\newblock In {\em Topology of {M}anifolds ({P}roc. {I}nst., {U}niv. of
  {G}eorgia, {A}thens, {G}a., 1969)}, pages 227--264. Markham, Chicago, Ill.,
  1970.

\bibitem[DZ79]{isomliegp}
J.~E. D'Atri and W.~Ziller.
\newblock Naturally reductive metrics and {E}instein metrics on compact {L}ie
  groups.
\newblock {\em Mem. Amer. Math. Soc.}, 18, 1979.

\bibitem[Haw53]{holsec1}
N.~Hawley.
\newblock Constant holomorphic curvature.
\newblock {\em Canad. J. Math.}, pages 53--56, 1953.

\bibitem[Hel78]{helgason}
S.~Helgason.
\newblock {\em Differential geometry, {L}ie groups, and symmetric spaces}.
\newblock Pure and Applied Mathematics. Academic Press, 1978.

\bibitem[Igu54]{holsec2}
J.-I. Igusa.
\newblock On the structure of a certain class of {K}aehler varieties.
\newblock {\em Amer. J. Math.}, 76(3):pp. 669--678, 1954.

\bibitem[JW80]{totgeodfol}
D.~Johnson and L.~Whitt.
\newblock Totally geodesic foliations.
\newblock {\em J. Diff. Geom.}, 15:225--235, 1980.

\bibitem[KN63]{KN63}
S.~Kobayashi and K.~Nomizu.
\newblock {\em Foundations of differential geometry}, volume~1.
\newblock Interscience, New York, 1963.

\bibitem[KN72]{kobnag}
S.~Kobayashi and T.~Nagano.
\newblock Riemannian manifolds with abundant isometries.
\newblock In {\em Differential geometry; in honor of {K}entaro {Y}ano}, pages
  195--219. Kinokuniya, 1972.

\bibitem[Kob72]{kotrgps}
S.~Kobayashi.
\newblock {\em Transformation Groups in Differential Geometry}.
\newblock Springer-Verlag, 1972.

\bibitem[Mok77]{mok}
K.P. Mok.
\newblock On the differential geometry of frame bundles of {R}iemannian
  manifolds.
\newblock {\em J. Reine Angew. Math.}, 302:16--31, 1977.

\bibitem[Mok79]{moklift}
K.P. Mok.
\newblock Complete lifts of tensor fields and connections to the frame bundle.
\newblock {\em Proc. London Math. Soc.}, 38(3):72--88, 1979.

\bibitem[MS39]{myerssteenrod}
S.~Myers and N.~Steenrod.
\newblock The group of isometries of a {R}iemannian manifold.
\newblock {\em Ann. of Math.}, (2):400--416, 1939.

\bibitem[O'N66]{oneillsasaki}
B.~O'Neill.
\newblock On the fundamental equations of a submersion.
\newblock {\em Michigan Math. J.}, 13:459--469, 1966.

\bibitem[Pet06]{petersen}
P.~Petersen.
\newblock {\em Riemannian geometry}.
\newblock Number 171 in Graduate Texts in Mathematics. Springer, 2006.

\bibitem[Sas58]{sasaki1}
S.~Sasaki.
\newblock On the differential geometry of tangent bundles of {R}iemannian
  manifolds.
\newblock {\em T{\^o}hoku Math. J.}, 10(3):338--354, 1958.

\bibitem[Sas62]{sasaki2}
S.~Sasaki.
\newblock On the differential geometry of tangent bundles of {R}iemannian
  manifolds {II}.
\newblock {\em T{\^o}hoku Math. J.}, 14(2):146--155, 1962.

\bibitem[TY91]{isomframe1}
H.~Takagi and M.~Yawata.
\newblock Infinitesimal isometries of frame bundles with a natural {R}iemannian
  metric.
\newblock {\em T{\^o}hoku Math. J.}, 43(1):1--148, 1991.

\bibitem[TY94]{isomframe2}
H.~Takagi and M.~Yawata.
\newblock Infinitesimal isometries of frame bundles with a natural {R}iemannian
  metric {II}.
\newblock {\em T{\^o}hoku Math. J.}, 46:341--355, 1994.

\bibitem[Wan47]{wanggaps}
H.C. Wang.
\newblock Finsler spaces with completely integrable equations of {K}illing.
\newblock {\em J. London Math. Soc.}, 22:5--9, 1947.

\bibitem[Zil]{zillerfat}
W.~Ziller.
\newblock Fatness revisited.
\newblock Lecture notes, University of Pennsylvania, 2001, available at
  http://www.math.upenn.edu/~wziller/papers/fatness.pdf.

\end{thebibliography}

\end{document}